\documentclass[a4paper, 10pt]{amsart}
\usepackage[utf8]{inputenc}
\usepackage[english]{babel}
\usepackage{hyperref}
\usepackage{amsmath, amsthm, amssymb, amsfonts}
\usepackage[foot]{amsaddr}
\usepackage{mathtools}
\usepackage{graphicx}
\usepackage{caption}
\usepackage{subcaption}
\usepackage{bm}
\usepackage{mathrsfs}
\usepackage{nicefrac}
\usepackage{tikz}
\hypersetup{colorlinks=true, linkcolor=blue, citecolor=blue}

\newtheorem{theorem}{Theorem}[section]

\newtheorem{proposition}[theorem]{Proposition}
\newtheorem{lemma}[theorem]{Lemma}

\theoremstyle{remark}
\newtheorem{remark}[theorem]{Remark}

\numberwithin{equation}{section}

\DeclareMathOperator{\image}{im}

\DeclareMathOperator{\dist}{dist}
\DeclareMathOperator{\id}{id}
\DeclareMathOperator{\imaginary}{Im}
\DeclareMathOperator{\real}{Re}

\newcommand{\R}{\mathbb{R}}
\newcommand{\Z}{\mathbb{Z}}
\newcommand{\N}{\mathbb{N}}
\newcommand{\C}{\mathbb{C}}
\newcommand{\ptorus}{\mathbb{S}_P}
\newcommand{\setsep}{\,;\,}
\newcommand{\doublehookrightarrow}{\mathrel{\mathrlap{{\mspace{4mu}\lhook}}{\hookrightarrow}}}
\newcommand{\jap}[2] {\langle#1\rangle^{#2}}
\newcommand{\bigjap}[2] {\big\langle#1\big\rangle^{#2}}
\newcommand{\biggjap}[2] {\bigg\langle#1\bigg\rangle^{#2}}
\newcommand{\bessel}[1]{\Lambda^{#1}}
\newcommand{\besselkernel}[1]{K_{#1}}
\newcommand{\besselkernelreg}[1]{J_{#1}}
\newcommand{\pbesselkernel}[1]{K_{P, #1}}
\newcommand{\FT}{\mathscr{F}}
\newcommand{\IFT}{\mathscr{F}^{-1}}
\newcommand{\holderspace}[1]{C^{#1}}
\newcommand{\holderspaceeven}[1]{C^{#1}_{\textnormal{even}}}
\newcommand{\zygmundspace}[1]{\mathcal{C}^{#1}}
\newcommand{\zygmundspaceeven}[1]{\mathcal{C}^{#1}_{\textnormal{even}}}
\newcommand{\schwartzspace}{\mathcal{S}}
\newcommand{\schwartzspacedual}{\mathcal{S}'}
\newcommand{\contspace}{BUC}
\newcommand{\contderspace}[1]{BUC^{#1}}

\newcommand{\norm}[1]{\|#1\|}
\newcommand{\seminorm}[1]{[#1]}
\newcommand{\diffop}[1]{\Delta_{#1}}
\renewcommand{\d}[1]{d#1}
\newcommand{\slot}{\,}
\newcommand{\set}[1]{\{#1\}}
\newcommand{\bigset}[1]{\big\{#1\big\}}
\newcommand{\lebesguespace}[1]{L^{#1}}
\newcommand{\bigparanth}[1]{\big(#1\big)}
\newcommand{\biggparanth}[1]{\bigg(#1\bigg)}
\newcommand{\integer}[1]{\lfloor#1\rfloor}
\newcommand{\fraction}[1]{\{#1\}}
\newcommand{\intconst}{\kappa}
\newcommand{\bigo}[1]{O(#1)}
\newcommand{\fouriercoeff}[2]{\hat{#1}_{#2}}
\newcommand{\abs}[1]{|#1|}
\newcommand{\bigabs}[1]{\big|#1\big|}

\newcommand{\wavespeed}{\mu}
\newcommand{\conv}{*}
\newcommand{\localbifurcationcurve}{\mathcal{R}}
\newcommand{\globalbifurcationcurve}{\mathfrak{R}}
\newcommand{\submaxset}{U}
\newcommand{\solutionset}{S}
\newcommand{\varphicoeff}[1]{\varphi_{#1}}

\newcommand{\wavespeedcoeff}[1]{\wavespeed_{#1}}
\newcommand{\cone}{\mathcal{K}}
\newcommand{\bifurcationfuncfDP}{G}
\newcommand{\trivialfuncfDP}{\widetilde{G}}
\newcommand{\constsol}{\gamma}
\newcommand{\shiftedsol}{\phi}
\newcommand{\localbifurcationcurvefDP}{\mathcal{Q}}
\newcommand{\localshiftedcurvefDP}{\widetilde{\mathcal{Q}}}
\newcommand{\globalbifurcationcurvefDP}{\mathfrak{Q}}
\newcommand{\submaxsetfDP}{V}
\newcommand{\solutionsetfDP}{W}
\newcommand{\perturbedsubmaxsetfDP}{\widetilde{V}}
\newcommand{\perturbedsolutionsetfDP}{\widetilde{W}}
\newcommand{\bifurcationpoint}{\wavespeed_{P, k}^*}
\newcommand{\eqindent}{\quad}

\newcommand{\remainder}[2]{R_{#1}^{#2}}
\newcommand{\boundary}{\partial}
\newcommand{\altsymbol}{m}
\newcommand{\altoperator}[1]{\widetilde{\Lambda}^{#1}}
\newcommand{\altkernel}[1]{\widetilde{K}_{#1}}
\newcommand{\paltkernel}[1]{\widetilde{K}_{P, #1}}
\newcommand{\constsolpos}{\gamma_+}
\newcommand{\constsolneg}{\gamma_-}

\title[Fractional Korteweg--De Vries and Degasperis--Procesi equations]{Highest waves for fractional Korteweg--De Vries and Degasperis--Procesi equations}

\date{}
\author[M. C. {\O}rke]{Magnus C. {\O}rke}
\email[]{magnusco@math.uio.no}

\begin{document}

\begin{abstract}
    We study traveling waves for a class of fractional Korteweg--De Vries and fractional Degasperis--Procesi equations with a parametrized Fourier multiplier operator of order $-s \in (-1, 0)$. For both equations there exist local analytic bifurcation branches emanating from a curve of constant solutions, consisting of smooth, even and periodic traveling waves. The local branches extend to global solution curves. In the limit we find a highest, cusped traveling-wave solution and prove its optimal $s$-Hölder regularity, attained in the cusp.
\end{abstract}

\maketitle

\section{Introduction} \label{sec:introduction}

We consider a class of fractional Korteweg--De Vries (fKdV) equations of the form
\begin{equation} \label{eq:main_fKdV}
    u_t + u u_x + (\bessel{-s} u)_x = 0, \qquad s \in (0, 1),
\end{equation}
and a class of fractional Degasperis--Procesi (fDP) equations similarly given by
\begin{equation} \label{eq:main_fDP}
    u_t + u u_x + \frac{3}{2} (\bessel{-s} u^2)_x = 0, \qquad s \in (0, 1),
\end{equation}
where $u(t, x)$ is a real-valued function, and the operator $\bessel{-s}$ is a Fourier multiplier defined as
\begin{equation*} \label{eq:bessel_operator_definition}
    \bessel{-s} \colon f \mapsto \IFT\bigparanth{\jap{\xi}{-s} \hat{f}(\xi)}, \qquad \jap{\xi}{-s} = (1 + \xi^2)^{-\frac{s}{2}},
\end{equation*}
in the sense of distributions (see \eqref{eq:fourier_transform_normalization} for our normalization of the Fourier transform). The nonlocal term $\bessel{-s} u$ can equivalently be characterized as a convolution ${\besselkernel{s} * u}$, with kernel
\begin{equation*} \label{eq:bessel_convolution_kernel}
	\besselkernel{s}(x) = \IFT(\jap{\xi}{-s})(x) = \frac{1}{2\pi} \int_{\R} \jap{\xi}{-s} e^{i x \xi} \slot \d{\xi}.
\end{equation*}

Inserting the traveling wave assumption ${u(x, t) = \varphi(x - \wavespeed t)}$ in the fKdV equation \eqref{eq:main_fKdV} and integrating, we obtain
\begin{equation} \label{eq:main_steady_fKdV}
	- \wavespeed \varphi + \frac{1}{2} \varphi^2 + \bessel{-s} \varphi = 0.
\end{equation}
The right-hand side of \eqref{eq:main_steady_fKdV} is assumed to be zero without loss of generality, due to the Galilean transformation
\begin{equation*}
    \varphi \mapsto \varphi + \gamma, \qquad \wavespeed \mapsto \wavespeed + \gamma,
\end{equation*}
with $\gamma$ chosen such that $\gamma(1-\wavespeed - \frac{1}{2} \gamma)$ cancels the possible constant of integration. Similarly, the traveling-wave assumption for the fDP equation yields
\begin{equation} \label{eq:main_steady_fDP}
    -\wavespeed \varphi + \frac{1}{2} \varphi^2 + \frac{3}{2} \bessel{-s} \varphi^2 = \intconst,
\end{equation}
but here it is not possible to obtain zero on the right-hand side while at the same time preserving the structure of the equation. Therefore, we work with an arbitrary real constant~$\intconst$ on the right-hand side in \eqref{eq:main_steady_fDP}.

When referring to a traveling-wave solution to the fKdV (resp. fDP) equation, we mean a real-valued continuous and bounded function $\varphi$ satisfying the equation \eqref{eq:main_steady_fKdV} (resp. \eqref{eq:main_steady_fDP}) on $\R$.

\subsection{Main results}

The goal of this paper is to characterize and prove existence of even and periodic traveling-wave solutions for the fKdV and the fDP equation. In particular, we shall see that traveling-wave solutions are cusped if the amplitude is equal to the wave speed~$\wavespeed$. We call such solutions highest traveling waves, and our main results deal with their existence and regularity:

\begin{theorem}\label{thm:alt_1_and_2_occur}
	Let $s \in (0, 1)$ and $P \in (0, \infty)$. Then there exists an $s$-Hölder continuous, $P$-periodic function $\varphi$, along with a number $\mu \in (0, 1)$, such that $\varphi$ is a highest traveling-wave solution to the steady fKdV equation \eqref{eq:main_steady_fKdV} with wave speed $\mu$. The solution $\varphi$ is even, strictly increasing and smooth on $(\nicefrac{-P}{2}, 0)$, with $\varphi(0) = \wavespeed$ and
	\begin{equation*}
		\mu - \varphi(x) \eqsim \abs{x}^s
	\end{equation*}
	uniformly for $\abs{x} \ll 1$.
\end{theorem}

Theorem \ref{thm:alt_1_and_2_occur} shows that there is a direct relationship between the order $-s$ of the dispersion in the equation and the Hölder regularity of exponent $s$ of the highest wave. This result is new, and expands upon the results from \cite{ehrnstrom_wahlen} for the Whitham equation 
\begin{equation} \label{eq:whitham_equation}
	u_t + uu_x + L_{W} u_x = 0, \qquad L_{W} \colon f \mapsto \IFT\bigparanth{\sqrt{\tanh{\xi}/ \xi}\, \hat{f}(\xi)},
\end{equation}
where the authors prove that there exist periodic highest traveling-wave solutions, which are cusped with exact $\nicefrac{1}{2}$-Hölder regularity at crests --- corresponding to the order $\nicefrac{-1}{2}$ of the dispersion in the equation. Our results thereby establish a broader picture of traveling waves for scalar, one-dimensional equations of the form \eqref{eq:whitham_equation} with weak inhomogeneous dispersive operators.

Using the same approach, we obtain the following result for the fDP equation.

\begin{theorem} \label{thm:alt_1_and_2_occur_fDP}
	Let $s \in (0, 1)$ and $\kappa > 0$. For small enough periods ${0 < P \ll 2\pi/ \sqrt{3^{2/s} - 1}}$, there exists an $s$-Hölder continuous, $P$-periodic function $\varphi$, along with a number $\mu > \sqrt{\intconst}$, such that $\varphi$ is a highest traveling-wave solution to the steady fDP equation \eqref{eq:main_steady_fDP} with wave-speed $\mu$. The solution $\varphi$ is even, strictly increasing and smooth on $(\nicefrac{-P}{2}, 0)$, with $\varphi(0) = \wavespeed$ and
	\begin{equation*}
		\mu - \varphi(x) \eqsim \abs{x}^s
	\end{equation*}
	uniformly for $\abs{x} \ll 1$.
\end{theorem}

This development is also new, demonstrating the same kind of connection between the order of the dispersion and the regularity of highest waves, even in the case when the nonlocal operator $\bessel{-s}$ acts on a quadratic term in $\varphi$. It is not obvious that such a result should hold for the fDP equation, since the balance between dispersion and nonlinear effects is changed compared to the fKdV equation --- in fact, that highest waves only exist for sufficiently small periods can be seen as a by-product of this adjusted balance.

\begin{remark}
	Throughout, $X \lesssim Y$ means that there is a positive constant $C$ such that the inequality ${X \leq C Y}$ holds. The relation $X \lesssim Y \lesssim X$ is denoted by $X \eqsim Y$. Writing $X \ll Y$ signifies that $X$ is "much smaller than" $Y$, or equivalently, that $X \leq cY$ some small positive constant $c$.
\end{remark}

\subsection{Background}

The present work stands in the context of several recent studies of nonlocal scalar models for surface water waves. An important example is the Witham equation~\eqref{eq:whitham_equation}, which was introduced in \cite{whitham_1967} by combining the structure of the KdV equation with the exact linear dispersion relation of gravity water waves. The model was motivated by physical considerations: as remarked by G. B. Whitham \cite[p. 476]{whitham}, nonlinear shallow water equations which neglect dispersion allow wave breaking but not traveling waves, while on the other hand the KdV equation allows traveling waves but not wave breaking. The dispersion in the Whitham equation is much weaker than that of the KdV equation, promoting a wider array of wave-phenomena than captured by either model on its own. Both wave-breaking~\cite{hur} and traveling waves \cite{ehrnstrom_wahlen, truong_wahlen_wheeler} have been proved for the Whitham equation, and as mentioned above, it was shown in \cite{ehrnstrom_wahlen} that there are highest periodic traveling-wave solutions which are exactly $\nicefrac{1}{2}$-Hölder regular in the cusp. It was furthermore conjectured in \cite{ehrnstrom_wahlen} that such solutions are convex between cusps and behave exactly like $\varphi(x) = \wavespeed - \sqrt{\frac{\pi}{2}} \abs{x}^{\frac{1}{2}} + o(x)$ (in the normalization of \eqref{eq:whitham_equation}); we refer to \cite{ehrnstrom_maehlen_varholm_2023, enciso_serrano_vergara} for results in this direction.

The fKdV and fDP equations can be seen as a toy models that we use to investigate the connection between the order of the operator $\bessel{-s}$ and the precise regularity of cusps of highest traveling waves. A partial result in this direction, for a class of generalized Whitham equations with a parametrized inhomogeneous symbol on the form $(\tanh(\xi)/\xi)^{s}$ of order in~$(-1, 0)$, is presented in \cite{afram}. The fKdV equation with a homogeneous symbol of order in~$(-1, 0)$ and a generalized nonlinearity has been studied in \cite{hildrum_xue}, where it is proved that the equation admits highest cusped traveling waves with Hölder regularity given exactly by the order of the dispersion. We point out that the symbol $(1 + \xi^2)^{-s/2}$ considered in this paper is inhomogeneous, and that we here take a different direction of generalization: instead of a generalized nonlinear (local) term we consider the case when the nonlocal operator $\bessel{-s}$ acts on a nonlinear term $u^2$, thereby broadening the analysis to fractional equations of Degasperis-Procesi type.

The steady fKdV equation \eqref{eq:main_steady_fKdV} with a parameter $s > 1$ has been studied in \cite{hung_le}, where highest periodic traveling waves with Lipschitz regularity at crests were proved to exist. In~\cite{bruell_dhara} the authors consider the homogeneous counterpart of the fKdV equation with $s > 1$, and analogous results are obtained. Dispersion of order corresponding to $s = 1$ has been considered in \cite{ehrnstrom_johnson_claassen} for a dispersive shallow water wave model; it was shown that highest cusped traveling waves of log-Lipschitz regularity exist.

The (local) Degasperis--Procesi equation was first studied in \cite{degasperis_procesi}, and is known to permit peaked traveling-wave solutions with Lipschitz regularity \cite{lenells}. A nonlocal formulation of the equation was studied in \cite{arnesen}, where the existence of highest periodic traveling waves of Lipschitz regularity at crests was proved.

\subsection{Outline}

In Section~\ref{sec:preliminaries} we recall properties of the Fourier multiplier $\bessel{-s}$ and its corresponding convolution kernel $\besselkernel{s}$. Section~\ref{sec:fkdv} treats the steady fKdV equation, first with a study of regularity of solutions and then existence by means of global analytic bifurcation. Our main contribution here is Theorem \ref{thm:regularity_2}, where we prove that highest traveling waves for the fKdV equation with a parameter $s$ are precisely \mbox{$s$-Hölder} regular at cusps. In Section~\ref{sec:fdp} we study regularity and existence for the steady fDP equation using the same framework. The main difficulty here is dealing with the nonlocal and nonlinear term $\bessel{-s}\varphi^2$. It turns out that this can be circumvented by rewriting the equation such that the nonlocal term is linear in $\varphi$, but with a slightly different structure in terms of the wave speed~$\wavespeed$, resulting in analogous regularity results but a different bifurcation pattern.

\section{Preliminaries} \label{sec:preliminaries}

We introduce conventions and study the nonlocal operators present in the fKdV and fDP equations. First, we show that the operator $\bessel{-s}$ is a smoothing operator on the scale of Hölder--Zygmund spaces, and present properties of the convolution kernel $\besselkernel{s}$. Second, we derive and study an additional Fourier multiplier operator $\altoperator{-s}$ defined by the symbol $\altsymbol(\xi) = 4 (3 + \jap{\xi}{s})^{-1}$.

\subsection{The operator \texorpdfstring{$\bessel{-s}$}{TEXT}} \label{subsec:convolution_kernel}

The Fourier transform is denoted by $\FT$ and defined on the Schwartz space $\schwartzspace(\R)$ of rapidly decreasing smooth functions on $\R$. It extends to the space of tempered distributions $\schwartzspacedual(\R)$ via duality. Our normalization is
\begin{equation} \label{eq:fourier_transform_normalization}
    (\FT f )(\xi) = \int_{\R} f(x) e^{-i x \xi} \slot \d{x}
\end{equation}
for $f \in \schwartzspace(\R)$, meaning that $(\IFT f)(x) = \frac{1}{2\pi} (\FT f)(-x)$. We sometimes write $\hat{f}$ for the Fourier transform of $f$

Let $\contspace(\R)$ denote the space of uniformly continuous and bounded functions over $\R$ normed by ${\norm{f}_{\infty} = \sup_{x \in \R} |f(x)|}$, and let $\contderspace{k}(\R)$ be the space of $k$ times uniformly continuous differentiable and bounded functions with the norm $\norm{f}_{C^{k}(\R)} = \sum_{m = 0}^{k} \norm{f^{(m)}}_{\infty}$. We use the convention that $\integer{\alpha}$ and $\fraction{\alpha}$ denote the integer and fractional part of the real number $\alpha > 0$, with $0 < \fraction{\alpha} \leq 1$ imposed. The space of $\alpha$-Hölder continuous functions on $\R$ with $\alpha \in (0, 1)$ is defined as
\begin{equation*} \label{eq:definition_holder_space}
    \holderspace{0, \alpha}(\R) = \set{f \in \contspace(\R) \setsep \seminorm{f}_{\holderspace{0, \alpha}(\R)} < \infty}, \qquad \seminorm{f}_{\holderspace{0, \alpha}(\R)} = \sup_{\substack{x, y \in \R\\ x \neq y}} \frac{|f(x) - f(y)|}{|x - y|^{\alpha}}.
\end{equation*}
Moreover, using the second-order difference
\begin{equation*}
	(\diffop{h}^2 f)(x) = (\diffop{h} (\diffop{h}f))(x) = f(x+2h) - 2f(x+h) + f(x),
\end{equation*}
we define for every $\alpha > 0$ the Zygmund (sometimes called Hölder-Besov) space
\begin{equation*} \label{eq:definition_zygmund_space}
    \zygmundspace{\alpha}(\R) = \set{f \in C^{\integer{\alpha}}(\R) \setsep \seminorm{f}_{\zygmundspace{\alpha}(\R)} < \infty}, \qquad \seminorm{f}_{\zygmundspace{\alpha}(\R)} = \sup_{0 \neq h \in \R} \frac{\norm{\diffop{h}^2 f^{(\integer{\alpha})}}_{C^{0}(\R)}}{|h|^{\fraction{\alpha}}}.
\end{equation*}
Then $\holderspace{\integer{\alpha}, \fraction{\alpha}}(\R)$ and $\zygmundspace{\alpha}(\R)$ are the spaces of $\integer{\alpha}$-times continuously differentiable bounded functions on $\R$ with Hölder and Zygmund exponent $\fraction{\alpha}$, respectively (see \cite{triebel_2} for details on how these spaces are defined). By \cite[Theorem 1.2.2]{triebel_2} the Hölder space $\holderspace{\integer{s}, \fraction{s}}$ and the Zygmund space $\zygmundspace{s}$ coincide for non-integer exponent $s$, in the sense of equivalent norms. In this context we sometimes refer to Hölder-Zygmund spaces, and the two are used interchangeably when there is no confusion.

By smoothing, we mean increasing the Hölder-Zygmund exponent. One can verify that
\begin{equation*} \label{eq:s_multiplier}
	\abs{D_\xi^k (1 + \xi^2)^{-\frac{s}{2}}} \lesssim_k (1 + \abs{\xi})^{-s-k}
\end{equation*}
for all $k \in \N_0$ (here, $\N_0$ is the set of nonnegative integers). It follows as a special case of \mbox{\cite[Proposition 2.78]{bahouri_chemin_danchin}} that the operator $\bessel{-s}$ is a linear and bounded map
\begin{equation} \label{eq:smoothing_operator}
	\bessel{-s}\colon L^{\infty}(\R) \rightarrow C^{0, s}(\R) \qquad \textnormal{and} \qquad \bessel{-s} \colon \zygmundspace{\alpha}(\R) \rightarrow \zygmundspace{\alpha + s}(\R)
\end{equation}
for every $\alpha > 0$ and $s \in (0, 1)$.

The following characterization of the kernel $\besselkernel{s}$ is a version of \cite[Proposition 1.2.5]{grafakos_modern}.

\begin{lemma} \label{lemma:asympt_kernel_behavior}
	Let $s \in (0, 1)$. Then
	\begin{itemize}
		\item[(i)] $\besselkernel{s}$ has the representation
		\begin{equation*} \label{eq:kernel_integral_representation}
			\besselkernel{s}(x) = \frac{1}{\sqrt{4 \pi} \Gamma(\frac{s}{2})} \int_{0}^{\infty} e^{-t -\frac{x^2}{4t}}\ t^{\frac{s - 3}{2}} \slot \d{t},
		\end{equation*}
		\item[(ii)] $\besselkernel{s}$ is even, strictly positive and smooth on $\R \setminus \set{0}$,
		\item[(iii)] there exist constants $C_s$ and $C_s'$ such that
		\begin{equation*} \label{eq:asympt_kernel_behavior}
			\begin{cases}
				\besselkernel{s}(x) \lesssim_s e^{-|x|} & |x| \geq 1, \\
				\besselkernel{s}(x) = C_s \abs{x}^{s - 1} + H_s(x) & |x| < 1,
			\end{cases}
		\end{equation*}
		where $H_s(x) = C_s' + \bigo{\abs{x}^{s+1}}$ with 
		\begin{equation*} \label{eq:derivatives_of_H_s}
			\abs{H_s'(x)} = \bigo{\abs{x}^s} \qquad \textnormal{and} \qquad \abs{H_s''(x)} = \bigo{\abs{x}^{s-1}}.
		\end{equation*}
	\end{itemize}
\end{lemma}

Furthermore, it turns out that $\besselkernel{s}$ is a completely monotone function. Recall that a smooth function ${g\colon (0, \infty) \rightarrow \R}$ is said to be completely monotone if
\begin{equation} \label{eq:completely_monotone_function}
    (-1)^n g^{(n)}(\lambda) \geq 0
\end{equation}
for all $n \in \N_0$ and $\lambda > 0$. This definition naturally extends to even functions which are smooth on $\R \setminus \set{0}$. The proof of the following proposition is based on \cite[Section 2]{ehrnstrom_wahlen}, while a more detailed account of completely monotone functions and related topics can be found in \cite{schilling_song_vondracek}.

\begin{proposition} \label{prop:completely_monotone_operator}
	For any $s \in (0, 1)$, the kernel $\besselkernel{s}$ is completely monotone. In particular, it is strictly decreasing and strictly convex on $(0, \infty)$.
\end{proposition}

\begin{proof}
	Let $h \colon (0, \infty) \rightarrow [0, \infty)$ be defined as $h(\lambda) = (1 + \lambda)^{-1}$. We claim that $h$ is a Stieltjes function, meaning that it can be written in terms of the integral representation
	\begin{equation*} \label{eq:stieltjes_function}
		h(\lambda) = \frac{a}{\lambda} + b + \int_{(0, \infty)} \frac{1}{\lambda + t} \d{\sigma(t)}, \qquad \textnormal{with} \qquad \int_{(0, \infty)} \frac{1}{1 + t} \d{\sigma(t)} < \infty
	\end{equation*}
	for the Borel measure $\sigma$ on $(0, \infty)$ and $a, b$ nonnegative constants. By \cite[Corollary 7.4]{schilling_song_vondracek}, if $g$ is a strictly positive function on $(0, \infty)$, then $g$ is Stieltjes if and only if
	\begin{equation*}
		\lim_{\lambda \searrow 0} g(\lambda) \in [0, \infty]
	\end{equation*}
	and $g$ has an analytic extension to $\C \setminus (-\infty, 0]$ with
	\begin{equation*}
		\imaginary(z) \imaginary(g(z)) \leq 0.
	\end{equation*}
	It is easy to verify that this is the case for the function $h$:
	\begin{equation*}
		\imaginary(\zeta) \imaginary(h(\zeta)) = -\frac{\imaginary(\zeta)^2}{(1+\real(\zeta))^2 + \imaginary(\zeta)^2} \leq 0
	\end{equation*}
	for every $\C \setminus (-\infty, 0]$. Hence, $h$ is Stieltjes. Even more, by \cite[Lemma 2.12]{ehrnstrom_wahlen} this means that $m^{s/2}$ is Stieltjes for every $s \in (0, 1)$.
	
	Now we invoke \cite[Proposition 2.20]{ehrnstrom_wahlen}, which states that if ${f \colon \R \rightarrow \R}$ and ${g \colon (0, \infty) \rightarrow \R}$ are two functions satisfying $f(\xi) = g(\xi^2)$ for all $\xi \neq 0$, then $f$ is the Fourier transform of an even, integrable, and completely monotone function if and only if $g$ is Stieltjes with
	\begin{equation*}
		\lim_{\lambda \searrow 0} g(\lambda) < \infty \qquad \textnormal{and} \qquad \lim_{\lambda \rightarrow \infty} g(\lambda) = 0.
	\end{equation*}
	This holds in our case, since we have shown that $h^{s/2}$ is Stieltjes and that
	\begin{equation*}
		\lim_{\lambda \searrow 0} h^{s/2} (\lambda) = 1 \qquad \textnormal{and} \qquad \lim_{\lambda \rightarrow \infty} h^{s/2} (\lambda) = 0
	\end{equation*}
	for every choice of $s \in (0, 1)$. But $h^{s/2}(\xi^2) = \jap{\xi}{-s}$ is the Fourier transform of $\besselkernel{s}$, meaning that $\besselkernel{s}$ is completely monotone.
	
	It remains to prove that $\besselkernel{s}$ is strictly decreasing and strictly convex on $(0, \infty)$. But according to \cite[Remark 1.5]{schilling_song_vondracek}, if $g$ is not identically constant, then \eqref{eq:completely_monotone_function} holds with strict inequality for every $\lambda$ and every $n \in \N_0$.
\end{proof}

Towards studying periodic solutions of the fKdV equation, we now define the periodic convolution kernel
\begin{equation*} \label{eq:periodic_kernel}
	\pbesselkernel{s}(x) = \sum_{n \in \Z} \besselkernel{s}(x + nP),
\end{equation*}
motivated by the observation
\begin{equation*}
		(\bessel{-s} f)(x) = (\besselkernel{s} \conv f)(x) = \int_{\R} \besselkernel{s}(x-y)f(y) \d{y} = \int_{-P/2}^{P/2} \pbesselkernel{s}(x-y)f(y) \d{y}
\end{equation*}
for every $P$-periodic smooth function $f$. Owing to Lemma~\ref{lemma:asympt_kernel_behavior}, one has
\begin{equation} \label{eq:periodic_singularity_kernel_representation}
	\pbesselkernel{s}(x) \eqsim_{P, s} |x|^{s-1} \qquad \textnormal{for} \qquad \abs{x} \ll 1.
\end{equation}
The previous discussion on $\besselkernel{s}$ implies that $\pbesselkernel{s}$ has the following properties (see \mbox{\cite[Remark 3.4]{ehrnstrom_wahlen}} for a proof).

\begin{lemma} \label{lemma:periodic_kernel}
	The periodic kernel $\pbesselkernel{s}$ is even, $P$-periodic and strictly increasing on $(\nicefrac{-P}{2}, 0)$.
\end{lemma}

Having established these properties of $\besselkernel{s}$ and $\pbesselkernel{s}$, the two subsequent lemmas follow by the same arguments as in \cite[Lemma 3.5, Lemma 3.6]{ehrnstrom_wahlen}.

\begin{lemma} \label{lemma:monotone_operator}
    Let $s \in (0, 1)$. If $f, g \in \contspace(\R)$ with $f \geq g$ and $f(x_0) > g(x_0)$ for some $x_0$, then
	\begin{equation*}
		\bessel{-s} f > \bessel{-s} g.
	\end{equation*}
\end{lemma}

\begin{lemma} \label{lemma:greater_than_zero_for_odd_functions}
	Let $s \in (0, 1)$ and $P \in (0, \infty]$. Assume that $f$ is an odd, $P$-periodic and continuous function with $f \geq 0$ on $(\nicefrac{-P}{2}, 0)$ and $f(x_0) > 0$ for some $x_0 \in (\nicefrac{-P}{2}, 0)$. Then
	\begin{equation*}
		\bessel{-s}f > 0
	\end{equation*}
	on $(\nicefrac{-P}{2}, 0)$.
\end{lemma}

\subsection{The operator \texorpdfstring{$\altoperator{-s}$}{TEXT}} \label{subsec:alternative_operator}

We derive a version of the steady fDP equation where the nonlocal operator acts on a linear term in $\varphi$. Taking the Fourier transform of the equation yields
\begin{equation*}
	\frac{1}{2} (1 + 3 \jap{\cdot}{-s}) \FT(\varphi^2) = \wavespeed \FT(\varphi) + \intconst \delta_0,
\end{equation*}
in distributional sense. Since $1 + 3 \jap{\xi}{-s}$ smooth and nonzero, this can be reformulated to
\begin{equation*}
	\frac{1}{2} \FT(\varphi^2) = \frac{1}{1 + 3 \jap{\cdot}{-s}} (\wavespeed \FT(\varphi) + \intconst \delta_0) = \wavespeed \FT(\varphi) - \frac{3}{4} \wavespeed \altsymbol(\cdot) \FT(\varphi) + \frac{1}{4} \intconst \delta_0,
\end{equation*}
where we have defined
\begin{equation*} \label{eq:altsymbol}
	\altsymbol(\xi) = \frac{4}{3 + \jap{\xi}{s}}.
\end{equation*}
Let $\altoperator{-s}$ be the Fourier multiplier defined by $\altsymbol$. Applying the inverse Fourier transform we arrive at
\begin{equation} \label{eq:main_steady_fdp_linearized}
    -\wavespeed \varphi + \frac{1}{2} \varphi^2 + \frac{3}{4} \wavespeed \altoperator{-s} \varphi = \frac{1}{4} \intconst,
\end{equation}
which with our assumptions on $\varphi$ may be understood in the strong, pointwise sense and is equivalent to the steady fDP equation.

Note that $\altoperator{-s}$ is a smoothing operator of order $-s$ (in the sense of \eqref{eq:smoothing_operator}). Indeed, using Faà di Bruno's formula it can be shown that ${\abs{m^{(k)}(\xi)} \lesssim_k (1 + \abs{\xi})^{-s-k}}$ for every $k \in \N_0$, and the claim follows again due to \cite[Proposition 2.78]{bahouri_chemin_danchin}.

As before, the operator $\altoperator{-s}$ can be written as a convolution with kernel and periodic kernel
\begin{equation*} \label{eq:altkernel}
	\altkernel{s}(x) = \IFT\bigparanth{m(\xi)}(x) \qquad \textnormal{and} \qquad \paltkernel{s}(x) = \sum_{n \in \Z} \altkernel{s}(x + nP).
\end{equation*}
Since $m(\xi)$ is smooth and all derivatives are integrable, we infer that $\altkernel{s}$ is smooth outside the origin and has rapidly decaying derivatives.

\begin{lemma} \label{lemma:asymptotic_altkernel_behavior}
	Let $s \in (0, 1)$. Then
	\begin{itemize}
		\item[(i)] $\altkernel{s}$ is even, nonnegative and integrable with $\norm{\altkernel{s}}_{\lebesguespace{1}(\R)} = 1$,
		\item[(ii)] $\paltkernel{s}$ is even, strictly increasing and smooth on $(\nicefrac{-P}{2}, 0)$, and
		\begin{equation} \label{eq:periodic_alt_kernel_singularity}
			\paltkernel{s} \eqsim_{P, s} \abs{x}^{s-1} \qquad \textnormal{for} \qquad \abs{x} \ll 1.
		\end{equation}
	\end{itemize}
\end{lemma}

\begin{proof}
	First we claim that the function
	\begin{equation*}
		\altsymbol(\sqrt{\abs{\xi}}) = \frac{4}{3 + (1 + \abs{\xi})^{s/2}}
	\end{equation*}
	is completely monotone on $(0, \infty)$. Indeed, it is a composition of functions ${g(y) = 4 (3 + y)^{-1}}$ and ${f(\xi) = (1 + \xi)^{s/2}}$, and one can check that $f$ is a Bernstein function (\mbox{\cite[Definition 3.1]{schilling_song_vondracek}}) and $g$ is completely monotone, so by \cite[Theorem 3.7]{schilling_song_vondracek} we conclude that $\altsymbol(\sqrt{\abs{\xi}})$ is completely monotone.

	(i) Clearly $\altkernel{s}$ is even and real since $m(\xi)$ is even and real. Next, recall the following two results due to Schoenberg and Bochner, respectively \cite{schilling_song_vondracek}. Firstly, a function $g\colon [0, \infty) \rightarrow \R$ continuous at zero is completely monotone if and only if $g(\abs{\cdot}^2)$ is positive definite on $\R^d$ for all $d \in \N$. Secondly, a function $f \colon \R^d \rightarrow \C$ is continuous and positive definite if and only if it is the Fourier transform of a finite nonnegative Borel measure on $\R^d$. This allows us to conclude that $m(\xi)$ is positive definite and consequently the Fourier transform of a finite nonnegative Borel measure. So $\altkernel{s}$ is integrable and nonnegative with
	\begin{equation*}
		\norm{\altkernel{s}}_{\lebesguespace{1}(\R)} = \int_{\R} \altkernel{s}(x) \slot \d{x} = (\IFT\FT(m))(0) = m(0) = 1.
	\end{equation*}
	
	(ii) It follows immediately from \cite[Theorem 2.5]{bruell_pei} that since $\altsymbol(\sqrt{\abs{\cdot}})$ is a completely monotone function and $m$ has the smoothing property, the periodic kernel is even, strictly increasing and smooth on $(\nicefrac{-P}{2}, 0)$. Let us assume first that $s > 1/2$, and write
	\begin{equation} \label{eq:first_expansion_m}
		\altsymbol(\xi) = 4\jap{\xi}{-s} - 12 \frac{\jap{\xi}{-2s}}{1 + 3 \jap{\xi}{-s}}.
	\end{equation}
	Taking the inverse Fourier transform and using Lemma~\ref{lemma:asympt_kernel_behavior} for the first term and that the second term is integrable, we obtain
	\begin{equation*}
		\abs{\IFT(m)(x)} \lesssim_s \abs{x}^{s - 1} \qquad \textnormal{for} \qquad \abs{x} \ll 1.
	\end{equation*}
	Similarly, for any $s \in (0, 1)$ we can continue the expansion \eqref{eq:first_expansion_m} and write $m(\xi)$ as a finite sum
	\begin{equation*}
		m(\xi) = 4 \sum_{n = 1}^{N-1} (-3)^{n-1} \jap{\xi}{-ns} + 4(-3)^{N-1} \frac{\jap{\xi}{-Ns}}{1 + 3 \jap{\xi}{-s}}
	\end{equation*}
	with $N-1 \leq 1/s < N$. Then the last term is integrable and we infer again by Lemma~\ref{lemma:asympt_kernel_behavior} that the singularity $\abs{x}^{s-1}$ dominates for small $\abs{x}$ (if $(N - 1)s = 1$ then the inverse Fourier transform gives a $\log$-term---see \cite[Proposition 1.2.5]{grafakos_modern}---but the conclusion still holds).
\end{proof}

Lemma~\ref{lemma:asymptotic_altkernel_behavior} implies in particular that $\altoperator{-s}$ is monotone in the same ways as $\bessel{-s}$:

\begin{lemma} \label{lemma:altoperator_is_monotone}
	Let $s \in (0, 1)$. If $f$ and $g$ are continuous and bounded functions on $\R$ with $f \geq g$ and $f(x_0) > g(x_0)$ for some $x_0$, then
	\begin{equation*}
		\altoperator{-s} f > \altoperator{-s} g.
	\end{equation*}
\end{lemma}

\begin{lemma} \label{lemma:altoperator_monotone_odd_fnc}
	Let $s \in (0, 1)$ and assume that $f$ is an odd, $P$-periodic and continuous function with $f \geq 0$ on $(\nicefrac{-P}{2}, 0)$ and $f(x_0) > 0$ for some $x_0 \in (\nicefrac{-P}{2}, 0)$. Then on $(\nicefrac{-P}{2}, 0)$ it holds
	\begin{equation*}
		(\altoperator{-s} f)(x) > 0.
	\end{equation*}
\end{lemma}

\section{The fKdV equation} \label{sec:fkdv}

In this section we prove existence of highest periodic traveling waves for the steady fKdV equation \eqref{eq:main_steady_fKdV} with parameter $s \in (0, 1)$ considered fixed throughout. In Section~\ref{subsec:traveling_waves_fkdv} we recover information about the magnitude and the sign of derivatives of solutions that satisfy certain periodicity and parity conditions. In Section~\ref{subsec:regularity_fkdv} it is proved that all solutions which have an amplitude strictly smaller than the wave-speed $\wavespeed$ are smooth, and that solutions which attain the maximal amplitude $\wavespeed$ are precisely $s$-Hölder continuous. Finally, existence of solutions by means of bifurcation is proved in Section~\ref{subsec:bifurcation_fkdv}.

We follow \cite{ehrnstrom_wahlen} regarding organization and methods. The main difference is that we here consider the parametrized operator $\bessel{-s}$, thereby obtaining a new relationship between the order of dispersion and the regularity of highest waves. Some results are stated for a period $P \in (0, \infty]$, where we adopt the convention that $P = \infty$ is the solitary case. The interval $[\nicefrac{-P}{2}, \nicefrac{P}{2}]$ with coinciding endpoints is denoted by $\ptorus$.

\subsection{Traveling-wave solutions} \label{subsec:traveling_waves_fkdv}

We begin with a proposition giving bounds for the minima and maxima of solutions. Recall that by a solution $\varphi$ we mean a real-valued continuous and bounded function satisfying the equation \eqref{eq:main_steady_fKdV} on $\R$. Note that
\begin{equation*}
	\bessel{-s} c = \besselkernel{s} \conv c = c \norm{\besselkernel{s}}_{\lebesguespace{1}} = c
\end{equation*}
for every constant $c \in \R$ and every $s \in (0, 1)$, so we have $\bessel{-s} \varphi \geq \bessel{-s} \min \varphi = \min \varphi$ and $\bessel{-s} \varphi \leq \bessel{-s} \max = \varphi \max \varphi$. Inserting this in the steady fKdV equation gives the following.

\begin{proposition} \label{prop:a_priori_inf_sup}
If $\varphi$ is a solution to the steady fKdV equation, then
\begin{equation*}
	\left\{
		\begin{alignedat}{2}
			& 2(\wavespeed - 1) \leq \min \varphi \leq 0 \leq \max \varphi\ \ \textnormal{ or }\ \ \varphi \equiv 2(\wavespeed - 1) \qquad \qquad & \textnormal{ if } \wavespeed \leq 1, \\
			& 0 \leq \min \varphi \leq 2(\wavespeed - 1) \leq \max \varphi\ \ \textnormal{ or }\ \ \varphi \equiv 0 & \textnormal{ if } \wavespeed > 1. \\
		\end{alignedat}
		\right.
	\end{equation*}
\end{proposition}

Any smooth, $P$-periodic function $f$ can be written as a uniformly convergent Fourier series
\begin{equation*} \label{eq:fourier_series}
    f(x) = \sum_{k \in \Z} \fouriercoeff{f}{k} e^{i \frac{2 \pi k}{P} x}, \qquad \textnormal{with} \qquad \fouriercoeff{f}{k} = \frac{1}{P} \int_{-P/2}^{P/2} f(x) e^{-i \frac{2 \pi k}{P} x}\slot \d{x}.
\end{equation*}
Fourier multipliers act on periodic functions by multiplying the Fourier coefficients of the function with the symbol of the operator. Precisely, the formula
\begin{equation*} \label{eq:fourier_mult_on_periodic_dist}
	\bessel{-s} f = \sum_{k \in \Z} \biggjap{\frac{2 \pi k}{P}}{-s} \fouriercoeff{f}{k} e^{i \frac{2 \pi k}{P} x}
\end{equation*}
is valid for any $f \in \schwartzspace(\ptorus)$ and extends to $\schwartzspacedual$ by duality. Thus, integrating the equation $\varphi^2 = 2\wavespeed \varphi - 2\bessel{-s} \varphi$ over $\ptorus$ we find that $\varphi \in \lebesguespace{2}(\ptorus)$ if $\varphi$ is integrable on $\ptorus$:

\begin{lemma} \label{lemma:a_priori_l2_norm}
	Let $P < \infty$. Then every solution $\varphi \in \lebesguespace{1}(\ptorus)$ to the steady fKdV equation belongs to $\lebesguespace{2}(\ptorus)$, with
	\begin{equation*}
		\norm{\varphi}_{\lebesguespace{2}(\ptorus)}^2 = 2 (\wavespeed - 1) \int_{\ptorus} \varphi \slot \d{x}.
	\end{equation*}
\end{lemma}

If a solution satisfies $\varphi(x_0) = 0$ at some point $x_0$, then evaluating the equation yields $(\bessel{-s} \varphi)(x_0) = 0$. Therefore, the solution $\varphi$ must either be identically equal to zero or it must change sign in $x_0$.

\begin{lemma} \label{lemma:nodal_properties}
    Let $P < \infty$. Every $P$-periodic, nonconstant and even solution $\varphi \in \contderspace{1}(\R)$ to the steady fKdV equation which is nondecreasing on $(\nicefrac{-P}{2}, 0)$ satisfies
    \begin{equation*}
        \varphi' > 0 \qquad \textnormal{and} \qquad \varphi < \wavespeed
    \end{equation*}
    on $(\nicefrac{-P}{2}, 0)$. If in addition $\varphi \in \contderspace{2}(\R)$, then $\varphi''(0) < 0$ and $\varphi''(\pm P/2) > 0$.
\end{lemma}

\begin{proof}
	Differentiating the steady fKdV equation yields
    \begin{equation*}
        (\wavespeed - \varphi) \varphi' = \bessel{-s} \varphi' > 0,
    \end{equation*}
    where the last inequality holds on $(\nicefrac{-P}{2}, 0)$ due to Lemma~\ref{lemma:greater_than_zero_for_odd_functions}. We conclude that $\varphi' > 0$ and $\varphi < \wavespeed$ on $(\nicefrac{-P}{2}, 0)$.
	
    Now assume that $\varphi \in \contderspace{2}(\R)$. Differentiating twice and evaluating in zero we get
    \begin{equation*}
        (\wavespeed - \varphi(0)) \varphi''(0) = (\bessel{-s} \varphi'')(0) = 2 \int_{0}^{P/2} \pbesselkernel{s}(y) \varphi''(y) \slot \d{y},
    \end{equation*}
    since $\varphi'(0) = 0$ by evenness and differentiability of $\varphi$, and because $\pbesselkernel{s}$ and $\varphi''$ are even functions. Then for $\varepsilon > 0$ is is possible to write
    \begin{equation*}
        \begin{aligned}
            \int_{0}^{P/2} \pbesselkernel{s}(y) \varphi''(y) \slot \d{y} = & \int_{0}^{\varepsilon} \pbesselkernel{s}(y) \varphi''(y) \slot \d{y} + \int_{\varepsilon}^{P/2} \pbesselkernel{s}(y) \varphi''(y) \slot \d{y} \\
            = & \int_{0}^{\varepsilon} \pbesselkernel{s}(y) \varphi''(y) \slot \d{y} + \bigg[ \pbesselkernel{s}(y) \varphi'(y) \bigg]_{y = \varepsilon}^{P/2} - \int_{\varepsilon}^{P/2} \pbesselkernel{s}'(y) \varphi'(y) \slot \d{y}
        \end{aligned}
    \end{equation*}
    (recall that $\pbesselkernel{s}$ is smooth outside of the origin). The first term vanishes when $\varepsilon \searrow 0$, because
	\begin{equation*}
		\lim_{\varepsilon \searrow 0}  \bigabs{\int_{0}^{\varepsilon} \pbesselkernel{s}(y) \varphi''(y) \slot \d{y}} \lesssim \norm{\varphi''}_{C(\R)} \lim_{\varepsilon \searrow 0} \int_{0}^{\varepsilon} \abs{y}^{s-1} \slot \d{y} = 0,
	\end{equation*}
	where we have used \eqref{eq:periodic_singularity_kernel_representation} for the period kernel. The second term must also vanish in the limit, since $\varphi'(P/2) = 0$, and since $\varphi'(\varepsilon) \lesssim \varepsilon$ due to $\varphi'(0) = 0$ and $\varphi' \in \contderspace{1}(\R)$. The last term is negative for each $\varepsilon > 0$, since we have proved both $\varphi' > 0$ and $\pbesselkernel{s}' > 0$ on $(\nicefrac{-P}{2}, 0)$. Hence, it is decreasing as $\varepsilon \searrow 0$ and so passing to the limit we arrive at
    \begin{equation*}
        (\wavespeed - \varphi(0)) \varphi''(0) = - 2 \lim_{\varepsilon \searrow 0} \int_{\varepsilon}^{P/2} \pbesselkernel{s}'(y) \varphi'(y) \slot \d{y} < 0.
    \end{equation*}
	That is, $\varphi''(0) < 0$ provided $\varphi < \wavespeed$. Arguing similarly as above, one has
	\begin{equation*}
		(\wavespeed - \varphi(P/2)) \varphi''(P/2) = 2 \bigg( \int_{0}^{P/2 - \varepsilon} + \int_{P/2 - \varepsilon}^{P/2} \bigg) \pbesselkernel{s}(P/2 + y) \varphi''(y) \slot \d{y},
	\end{equation*}
	where the second term vanishes when $\varepsilon \searrow 0$. The first term can be integrated by parts, and passing to the limit we obtain
	\begin{equation*}
		(\wavespeed - \varphi(P/2)) \varphi''(P/2) = - 2\lim_{\varepsilon \searrow 0} \int_{0}^{P/2 - \varepsilon} \pbesselkernel{s}'(P/2 + y) \varphi'(y) \slot \d{y} > 0,
	\end{equation*}
	on account of $K'_{P, s}$ being $P$-periodic and strictly positive on $(\nicefrac{-P}{2}, 0)$, and $\varphi'$ strictly negative on $(0, \nicefrac{P}{2})$. Hence, ${\varphi''(P/2) > 0}$, and by evenness also $\varphi''(-P/2) > 0$.
\end{proof}

The previous lemma does not hold without the assumption of differentiability. Supposing instead that $\varphi \leq \wavespeed$, one can check that an argument analogous to that of \cite[Lemma 5.2]{ehrnstrom_wahlen} implies the following.

\begin{lemma} \label{lemma:smooth_away_from_crest}
	Let $P \in (0, \infty]$. Assume that $\varphi$ is an even, $P$-periodic and nonconstant solution to the steady fKdV equation which is nondecreasing on $(\nicefrac{-P}{2}, 0)$ with $\varphi \leq \wavespeed$. Then $\varphi$ is strictly increasing on $(\nicefrac{-P}{2}, 0)$.
\end{lemma}

\subsection{Regularity of solutions} \label{subsec:regularity_fkdv}

Writing the fKdV equation in the form
\begin{equation} \label{eq:bootstrap_form}
	\varphi = \wavespeed - \sqrt{\wavespeed^2 - 2\bessel{-s} \varphi},
\end{equation}
a bootstrapping argument along the lines of \cite[Theorem 5.1]{ehrnstrom_wahlen}, using that the operator $\bessel{-s}$ is linear and bounded from $L^{\infty}(\R)$ to $\zygmundspace{s}(\R)$ and from $\zygmundspace{\alpha}(\R)$ to $\zygmundspace{\alpha + s}(\R)$, can be used to prove that $\varphi$ is smooth wherever it is strictly below the maximal wave-height:

\begin{lemma} \label{lemma:regularity_1}
	Let $\varphi \leq \wavespeed$ be a solution to the steady fKdV equation. Then $\varphi$ is smooth on every open set where $\varphi < \wavespeed$.
\end{lemma}

Thus, solutions $\varphi$ satisfying the assumptions of Lemma~\ref{lemma:smooth_away_from_crest} are smooth except possibly in $x = 0$, where smoothness may break down provided that $\varphi(0) = \wavespeed$. As the next lemma shows, such solutions are not continuously differentiable in $x = 0$; in fact, they can be at most $s$-Hölder regular, where $s$ is the dispersion parameter appearing in $\bessel{-s}$.

\begin{proposition} \label{prop:upper_regularity_bound}
	Let $P \in (0, \infty]$. Assume that $\varphi$ is an even, $P$-periodic and nonconstant solution to the steady fKdV equation which is nondecreasing on $(\nicefrac{-P}{2}, 0)$ with $\varphi \leq \wavespeed$. Then
	\begin{equation} \label{eq:lower_bound_crest}
		\wavespeed - \varphi(x) \gtrsim |x|^{s}
	\end{equation}
	uniformly for $|x| \ll 1$. Moreover, if $P < \infty$ then
	\begin{equation} \label{eq:lower_bound_end_interval}
		\wavespeed - \varphi(-P/2) \gtrsim 1.
	\end{equation}	
\end{proposition}

\begin{proof}
	Assume first that $P < \infty$, and note that for every $h \in (0, \nicefrac{P}{2})$ we have the formula
	\begin{equation} \label{eq:first_symmetrization}
		\begin{aligned}
			&(\bessel{-s} \varphi)(x + h) - (\bessel{-s} \varphi)(x - h) \\
			& = \int_{-P/2}^0 (\pbesselkernel{s}(x - y) - \pbesselkernel{s}(x + y))(\varphi(y + h) - \varphi(y - h)) \slot \d{y}.
		\end{aligned}
	\end{equation}
	 Since $\varphi$ is smooth except possibly in $x = 0$, one has for $x \in (\nicefrac{-P}{2}, 0)$ that
	\begin{equation*}
		\begin{aligned}
			(\wavespeed - \varphi(x)) \varphi'(x) & = (\bessel{-s} \varphi)'(x) \\
			& = \lim_{h \rightarrow 0} \frac{((\bessel{-s} \varphi)(x + h) - (\bessel{-s} \varphi)(x - h))}{2h} \\
			& \geq \liminf_{h \rightarrow 0} \frac{1}{2h} \int_{-P/2}^0 (\pbesselkernel{s}(x - y) - \pbesselkernel{s}(x + y))(\varphi(y + h) - \varphi(y - h)) \slot \d{y} \\
			& \geq \int_{-P/2}^0 (\pbesselkernel{s}(x - y) - \pbesselkernel{s}(x + y)) \varphi'(y) \slot \d{y},
		\end{aligned}
	\end{equation*}
	where we used the formula \eqref{eq:first_symmetrization} in the third step and differentiation under the integral is justified by Fatou's lemma. Fix ${x_0 \in (\nicefrac{-P}{2}, 0)}$ and let $x \in (\frac{x_0}{2}, \frac{x_0}{4})$. Then, with ${z \in [\nicefrac{-P}{2}, x]}$, we have
	\begin{equation} \label{eq:lower_bound_starting_point}
		\begin{aligned}
			(\wavespeed - \varphi(z)) \varphi'(x) & \geq (\wavespeed - \varphi(x)) \varphi'(x) \\
			& \geq \int_{-P/2}^0 (\pbesselkernel{s}(x - y) - \pbesselkernel{s}(x + y)) \varphi'(y) \slot \d{y} \\
			& \geq \int_{x_0/2}^{x_0/4} (\pbesselkernel{s}(x - y) - \pbesselkernel{s}(x + y)) \varphi'(y) \slot \d{y},
		\end{aligned}
	\end{equation}
	since the integrand is strictly positive. Letting
	\begin{equation*}
		C_{P} = \min \set{\pbesselkernel{s}(x - y) - \pbesselkernel{s}(x + y) \setsep x, y \in (\frac{x_0}{2}, \frac{x_0}{4})} > 0,
	\end{equation*} 
	we have
	\begin{equation*}
		(\wavespeed - \varphi(-P/2)) \varphi'(x) \geq C_{P} (\varphi(\frac{x_0}{4}) - \varphi(\frac{x_0}{2})).
	\end{equation*}
	Integrating over $(\frac{x_0}{2}, \frac{x_0}{4})$ and dividing by the difference $\varphi(x_0/4) - \varphi(x_0/2)$ we obtain 
	\begin{equation*}
		(\wavespeed - \varphi(-\frac{P}{2})) \geq \frac{1}{4} C_{P} \abs{x_0}
	\end{equation*}
	and thus \eqref{eq:lower_bound_end_interval} by choosing $x_0 = -P/4$, say.
	
	Towards proving \eqref{eq:lower_bound_crest}, note that by the mean value theorem and \eqref{eq:periodic_singularity_kernel_representation} we have
	\begin{equation*}
		\pbesselkernel{s}(x - y) - \pbesselkernel{s}(x + y) \geq -2y \pbesselkernel{s}'\bigparanth{x_0} \gtrsim \abs{x_0}^{s-1}
	\end{equation*}
	uniformly over $x, y \in (x_0/2, x_0/4)$ with $\abs{x_0} \ll 1$. Inserting the above in \eqref{eq:lower_bound_starting_point} yields
	\begin{equation*}
		(\wavespeed - \varphi(z)) \varphi'(x) \gtrsim |x_0|^{s-1} (\varphi(x_0/4) - \varphi(x_0/2)).
	\end{equation*}
	Integrating this inequality over $(x_0/2, x_0/4)$ with respect to $x$, dividing by the (positive) difference $(\varphi(x_0/4) - \varphi(x_0/2)$, and setting $z = x_0$, we obtain
	\begin{equation*}
		(\wavespeed - \varphi(x_0) \gtrsim (x_0/4 - x_0/2) |x_0|^{s - 1} \gtrsim |x_0|^s,
	\end{equation*}
	uniformly for $|x_0| \ll 1$. The estimate \eqref{eq:lower_bound_crest} now follows by evenness of $\varphi$. Moreover, \eqref{eq:lower_bound_crest} holds in the solitary case $P = \infty$ as well, since the estimate can be chosen uniformly for large $P$, and in the limit one can use the same properties for $\besselkernel{s}$.
\end{proof}

Proposition~\ref{prop:upper_regularity_bound} provides an upper bound for the regularity of solutions which attains the value $\wavespeed$ from below in $x = 0$. In Theorem~\ref{thm:regularity_2} we prove that solutions are at least globally $s$-Hölder regular, with the precise regularity  attained in $x = 0$ in the case $\varphi(0) = \wavespeed$.

\begin{theorem} \label{thm:regularity_2}
	Let $P \in (0, \infty]$, and let $\varphi \leq \wavespeed$ be an even and nonconstant solution to the steady fKdV equation which is nondecreasing on $(\nicefrac{-P}{2}, 0)$ and with $\varphi(0) = \wavespeed$. Then $\varphi \in C^{0, s}(\R)$. Moreover,
	\begin{equation} \label{eq:exact_sigma_holder_at_crest}
		\wavespeed - \varphi(x) \eqsim |x|^s
	\end{equation}
	uniformly for $\abs{x} \ll 1$.
\end{theorem}

\begin{proof}
    We show first that the solution $\varphi$ is $\alpha$-Hölder continuous in $0$ for every $\alpha < s$. Taking the difference between the steady fKdV equation evaluated in two points $x$ and $y$, we obtain the formula
	\begin{equation} \label{eq:point_comparison_form}
		(2 \wavespeed - \varphi(y) - \varphi(x))(\varphi(y) - \varphi(x)) = 2\bigparanth{(\bessel{-s}\varphi)(y) - (\bessel{-s}\varphi)(x)},
	\end{equation}
	and using that $\varphi(0) = \wavespeed$ this can be written as
    \begin{equation} \label{eq:holder_bootstrap_in_0}
        \begin{aligned}
            (\wavespeed - \varphi(x))^2 & = 2 \bigparanth{(\bessel{-s}\varphi)(0) - (\bessel{-s}\varphi)(x)} \\
            & = \int_{\R} (\besselkernel{s}(x + y) + \besselkernel{s}(x - y) - 2 \besselkernel{s}(y)) (\varphi(0) - \varphi(y)) \slot \d{y}.
        \end{aligned}
	\end{equation}
	Owing to Lemma~\ref{lemma:asympt_kernel_behavior} the kernel $\besselkernel{s}$ may be split into singular and regular parts according to
	\begin{equation} \label{eq:singular_regular_kernel}
		\besselkernel{s}(x) = C_s \abs{x}^{s-1} + \besselkernelreg{s}(x),
	\end{equation}
	where $\besselkernelreg{s}(x)$ is continuously differentiable with
	\begin{equation} \label{eq:k_reg_first_derivative}
		\abs{\besselkernelreg{s}'(x)} \lesssim (1 + \abs{x})^{s-2}
	\end{equation}
	and furthermore
	\begin{equation} \label{eq:k_reg_second_derivative}
		\left\{
			\begin{alignedat}{2}
				\abs{\besselkernelreg{s}''(x)} & = \bigo{\abs{x}^{s-1}} && \qquad \abs{x} < 1, \\
				\abs{\besselkernelreg{s}''(x)} & \lesssim (1 + \abs{x})^{s-3} && \qquad \abs{x} \geq 1.
			\end{alignedat}
		\right.
	\end{equation}
	Note that by the mean value theorem,
	\begin{equation*} \label{eq:first_difference_regular_kernel}
		\abs{\besselkernelreg{s}(y + x) - \besselkernelreg{s}(y)} \leq \abs{x} \int_0^1 \abs{\besselkernelreg{s}'(y + tx)} \slot \d{t} = \abs{x} \remainder{x}{1}(y)
	\end{equation*}
	where $\remainder{x}{1}(y)$ denotes the integral part. Similarly, we have
	\begin{equation*} \label{eq:second_order_kernel}
		\abs{\besselkernelreg{s}(y+x) + \besselkernelreg{s}(y-x) - 2\besselkernelreg{s}(y)} \leq \abs{x}^2 \int_0^1 \int_0^1 2t \abs{\besselkernelreg{s}''(y - tx + 2rtx)} \slot \d{r} \d{t} = \abs{x}^2 \remainder{x}{2}(y).
	\end{equation*}
	We insert \eqref{eq:singular_regular_kernel} in \eqref{eq:holder_bootstrap_in_0} and estimate each part. For the singular term one has
	\begin{equation} \label{eq:singular_part_first_bootstrap}
		\begin{aligned}
			& C_s \int_{\R} \bigabs{\abs{x + y}^{s-1} + \abs{x - y}^{s-1} - 2 \abs{y}^{s-1}} (\varphi(0) - \varphi(y)) \slot \d{y} \\
			& \leq 2 C_s \norm{\varphi}_{\lebesguespace{\infty}} \abs{x}^s \int_{\R} \bigabs{\abs{1 + t}^{s-1} + \abs{1 - t}^{s-1} - 2 \abs{t}^{s-1}} \slot \d{t} \lesssim \abs{x}^s,
		\end{aligned}
	\end{equation}
	where we have used that the integral in the last step converges for every $s \in (0, 1)$ since the integrand is $\bigo{\abs{t}^{s-3}}$ as $\abs{t} \to \infty$. The regular part can be estimated by
	\begin{equation} \label{eq:regular_part_first_bootstrap}
		\begin{aligned}
			& \int_{\R} \bigabs{\besselkernelreg{s}(x + y) + \besselkernelreg{s}(x - y) - 2 \besselkernelreg{s}(y)} (\varphi(0) - \varphi(y)) \slot \d{y} \\
			& \lesssim \norm{\varphi}_{\lebesguespace{\infty}} \abs{x}^2 \int_{\R} \remainder{x}{2}(y) \slot \d{y} \lesssim \abs{x}^2,
		\end{aligned}
	\end{equation}
	where the integral of $\remainder{x}{2}(y)$ is uniformly bounded for $\abs{x} \ll 1$ in view of \eqref{eq:k_reg_second_derivative}. Inserting \eqref{eq:singular_part_first_bootstrap} and \eqref{eq:regular_part_first_bootstrap} in \eqref{eq:holder_bootstrap_in_0} yields $(\wavespeed - \varphi(x))^2 \lesssim \abs{x}^s$. This implies that $\varphi$ is at least $\frac{s}{2}$-Hölder continuous in ${x=0}$. Using this information, the term $\varphi(0) - \varphi(y)$ can now be bounded from above by $|y|^{\frac{s}{2}}$ in \eqref{eq:singular_part_first_bootstrap}, giving $\frac{s/2 + s}{2}$-Hölder continuity of $\varphi$ in ${x=0}$ in the same way. Iterating this argument proves that $\varphi$ is $\alpha$-Hölder regular in ${x=0}$ for every $\alpha < s$.

    We show $s$-Hölder regularity in $x = 0$. To this end, we claim that there is a constant $C$ which is independent of $\alpha$ such that
	\begin{equation*}
		\int_{\R} \bigabs{\besselkernel{s}(x + y) + \besselkernel{s}(x - y) - 2 \besselkernel{s}(y)} |y|^{\alpha} \slot \d{y} \leq C |x|^{2 \alpha}
	\end{equation*}
	for all $|x| \leq 1$ and all $0 \leq \alpha \leq s$. Indeed, for the singular part we have
	\begin{equation*}
		\begin{aligned}
			& C_s \int_{\R} \bigabs{\abs{x + y}^{s-1} + \abs{x - y}^{s-1} + 2 \abs{y}^{s-1}} \abs{y}^\alpha \slot \d{y} \\
            & = C_s \abs{x}^{s + \alpha} \int_{\R} \bigabs{|1 + t|^{s-1} + |1 - t|^{s-1} - 2 |t|^{s-1}} |t|^{\alpha} \slot \d{t} \lesssim \abs{x}^{s + \alpha} \leq \abs{x}^{2\alpha},
		\end{aligned}
	\end{equation*}
	uniformly for $\alpha \in [0, s]$, where in the last step it was used that $\abs{x} \leq 1$. Moreover, the regular part of the kernel can be bounded according to
    \begin{equation*}
            \int_{\R} \bigabs{\besselkernelreg{s}(x + y) + \besselkernelreg{s}(x - y) - 2 \besselkernelreg{s}(y)} |y|^\alpha \slot \d{y} \leq |x|^2 \int_{\R} \remainder{x}{2}(y) |y|^\alpha \slot \d{y} \lesssim \abs{x}^2,
    \end{equation*}
	and for $\abs{x} \leq 1$ we have $\abs{x}^2 \leq \abs{x}^{2\alpha}$. It was shown above that $\varphi$ is $\alpha$-Hölder continuous in the origin for every $\alpha \in [0, s)$. Hence,
	\begin{equation*}
		\begin{aligned}
			(\varphi(0) - \varphi(x))^2 & = \int_{\R} \big( \besselkernel{s}(x + y) + \besselkernel{s}(x - y) - 2 \besselkernel{s}(y) \big) (\varphi(0) - \varphi(y)) \slot \d{y} \\
			& \leq \seminorm{\varphi}_{\holderspace{0, \alpha}_0} \int_{\R} \bigabs{\besselkernel{s}(x + y) + \besselkernel{s}(x - y) - 2 \besselkernel{s}(y)} |y|^\alpha \slot \d{y} \\
			& \lesssim \seminorm{\varphi}_{\holderspace{0, \alpha}_0} |x|^{2 \alpha},
		\end{aligned}
	\end{equation*}
	where 
	\begin{equation*}
		\seminorm{\varphi}_{\holderspace{0, \alpha}_0(\R)} = \sup_{\substack{h \in \R\\ h \neq 0}} \frac{|\varphi(h) - \varphi(0)|}{|h|^{\alpha}}.
	\end{equation*}
	Dividing by $|x|^{2 \alpha}$ and passing to supremum yields $\seminorm{\varphi}_{\holderspace{0, \alpha}_0} \lesssim 1$ uniformly over $\alpha \in [0, s)$. We let $\alpha \nearrow s$, and combined with \eqref{eq:lower_bound_crest} this proves \eqref{eq:exact_sigma_holder_at_crest}.

    As in \cite{ehrnstrom_wahlen}, to prove global $\alpha$-Hölder regularity for some $\alpha \in (0, 1)$ it suffices to show that
	\begin{equation*} \label{eq:global_holder_regularity_reduction}
		\sup_{0<h<\abs{x}<\delta} \frac{|\varphi(x + h) - \varphi(x - h)|}{h^\alpha} < \infty
	\end{equation*}
	for some $\delta > 0$ (recall that $\varphi(x + y) - \varphi(x - y)$ is symmetric in $x$ and $y$ and $\varphi$ is smooth outside of the origin). We proceed to show that $\varphi \in \holderspace{0, \alpha}(\R)$ for every $\alpha < s$. So assume that $0<h<x<\delta$ for some $\delta \ll 1$, where $x$ can be taken positive without loss of generality. Since
	\begin{equation} \label{eq:global_bootstrap}
		\begin{aligned}
			& (\varphi(x + h) - \varphi(x - h))^2 \\
			& \leq \bigabs{(2\wavespeed - \varphi(x + h) - \varphi(x - h))(\varphi(x + h) - \varphi(x - h))} \\
			& = 2 \bigabs{(\bessel{-s}\varphi)(x + h) - (\bessel{-s}\varphi)(x - h)},
		\end{aligned}
	\end{equation}
	and $\bessel{-s}$ maps $\lebesguespace{\infty}$ to $\holderspace{0, s}$ and $\zygmundspace{\alpha}$ to $\zygmundspace{\alpha + s}$, we obtain that $\varphi$ is at least $\alpha$-Hölder regular for every $\alpha < s$ if $s \leq 1/2$ and $\alpha = 1/2$ if $s > 1/2$. Consequently, for $s > 1/2$ we need to pass the threshold $\alpha = 1/2$ in the iteration procedure of \eqref{eq:global_bootstrap}. So assume that $s > 1/2$ and that $\varphi \in \holderspace{0, \alpha}$ with $\alpha + s > 1$. Note that for a function $f \in \holderspace{1, \beta}$ with $\beta \in (0, 1)$ and $f'(0) = 0$, one has
	\begin{equation*} \label{eq:holder_larger_than_one}
		\abs{f(x) - f(y)} = \abs{x-y} \abs{f'(\zeta) - f'(0)} \lesssim \abs{x-y} \abs{\zeta}^{\beta}
	\end{equation*}
	for some $\zeta \in (x, y)$. Hence,
	\begin{equation*}
		\bigabs{(\bessel{-s}\varphi)(x + h) - (\bessel{-s}\varphi)(x - h)} \lesssim h \abs{\zeta}^{\fraction{\alpha + s}},
	\end{equation*}
	with $\zeta \in (x-h, x+h)$ and $\fraction{\alpha + s}$ being the fractional part of $\alpha + s$. Inserting this in \eqref{eq:global_bootstrap} yields
	\begin{equation} \label{eq:first_interpolation_nonsharp}
		\begin{aligned}
			\abs{\varphi(x + h) - \varphi(x - h)} & \lesssim \frac{h \abs{\zeta}^{\fraction{\alpha + s}}}{2\wavespeed - \varphi(x + h) - \varphi(x - h)} \\
			& \lesssim \frac{h \abs{x+h}^{\fraction{\alpha + s}}}{\abs{x+h}^s + \abs{x-h}^s} \\
			& \lesssim h \abs{x+h}^{\alpha - 1}
		\end{aligned}
	\end{equation}
	where we have used the estimate \eqref{eq:lower_bound_crest} from Proposition~\ref{prop:upper_regularity_bound} in the second step, and in the last step that $\fraction{\alpha + s} - s = \alpha - 1$. Now we interpolate between \eqref{eq:first_interpolation_nonsharp} and the exact $s$-Hölder regularity in the origin. Precisely, with $\sigma \in (0, 1)$ one has
	\begin{equation*}
		\begin{aligned}
			\frac{\abs{\varphi(x + h) - \varphi(x - h)}}{h^\sigma} & \leq \frac{\abs{\varphi(x + h) - \varphi(x - h)}^{\sigma}}{h^\sigma} \abs{\wavespeed - \varphi(x + h)}^{1 - \sigma} \\
			& \lesssim \abs{x + h}^{\sigma(\alpha - 1) + (1 - \sigma)s}.
		\end{aligned}
	\end{equation*}
	This is bounded whenever
	\begin{equation*}
		\sigma \leq \frac{s}{1+s-\alpha},
	\end{equation*}
	and we choose the interpolation parameter $\sigma$ such that equality holds. Hence,
	\begin{equation*}
		\abs{\varphi(x + h) - \varphi(x - h)} \lesssim h^{\frac{s}{1+s-\alpha}}.
	\end{equation*}
	Iterating this argument, one obtains in each step for $\varphi \in \holderspace{0, \alpha}$ that $\varphi$ is $\frac{s}{1+s-\alpha}$-Hölder regular. The regularity is therefore increased in each iteration and tending to $s$, proving $\varphi \in \holderspace{0, \alpha}(\R)$ for every $\alpha < s$.
    
    We now prove $\varphi \in \holderspace{0, s}(\R)$. To this end, note that the difference in the right-hand side of \eqref{eq:point_comparison_form} can also be written as
	\begin{equation} \label{eq:second_symmetry_identity}
			(\bessel{-s}\varphi)(x + h) - (\bessel{-s}\varphi)(x - h) = \int_{-\infty}^{0} ( \besselkernel{s}(y + h) - \besselkernel{s}(y - h)) (\varphi(y - x) - \varphi(y + x) \slot \d{y}.
	\end{equation}
	Let $0<h<x<\delta$ for some $\delta \ll 1$. Since
	\begin{equation*}
		2\wavespeed - \varphi(x + h) - \varphi(x - h) \geq \wavespeed - \varphi(x + h) \geq \wavespeed - \varphi(x),
	\end{equation*}
	we have with \eqref{eq:point_comparison_form} and \eqref{eq:second_symmetry_identity} that
	\begin{equation} \label{eq:interpolation_startingpoint}
		\begin{aligned}
			& (\wavespeed - \varphi(x)) \abs{\varphi(x + h) - \varphi(x - h)} \\
			& \leq 2 \int_{-\infty}^0 \abs{\besselkernel{s}(y + h) - \besselkernel{s}(y - h)} \abs{\varphi(y - x) - \varphi(y + x)} \slot \d{y}.
		\end{aligned}
	\end{equation}
	To estimate the factor $\abs{\varphi(y - x) - \varphi(y + x)}$, we interpolate between the global $\holderspace{0, \alpha}$-regularity (for $\alpha < s$) and the sharp $\holderspace{0, s}$-regularity in $x = 0$. That is, between
	\begin{equation} \label{eq:interpolation_from_global}
		|\varphi(y - x) - \varphi(y + x)| \lesssim \norm{\varphi}_{\holderspace{0, \alpha}} \min(x^{\alpha}, \abs{y}^{\alpha})
	\end{equation}
	for every choice of $\alpha \in (0, s)$, and
	\begin{equation} \label{eq:interpolation_from_local}
		\begin{aligned}
			\abs{\varphi(y - x) - \varphi(y + x)} & \leq  \abs{\wavespeed - \varphi(y - x)} + \abs{\wavespeed - \varphi(y + x)} \\
			& \lesssim \seminorm{\varphi}_{\holderspace{0, s}_0} \max(x^{s}, \abs{y}^{s}).
		\end{aligned}
	\end{equation}
	Interpolation of \eqref{eq:interpolation_from_global} and \eqref{eq:interpolation_from_local} over a parameter $\eta$ gives
	\begin{equation} \label{eq:interpolation}
		\abs{\varphi(y - x) - \varphi(y + x)} \lesssim \norm{\varphi}_{\holderspace{0, \alpha}}^\eta \min(x, \abs{y})^{\alpha \eta} \max(x, \abs{y})^{s(1 - \eta)},
	\end{equation}
	with $(\alpha, \eta) \in (0, s) \times [0, 1]$. The integral in the right-hand side of \eqref{eq:interpolation_startingpoint} can be split in the singular and regular parts of the kernel $\besselkernel{s}$. Inserting \eqref{eq:interpolation} in the integral with the singular term yields
	\begin{equation} \label{eq:singular_global_estimate}
		\begin{aligned}
			& C_s \int_{-\infty}^0 \bigabs{\abs{y + h}^{s-1} - \abs{y - h}^{s-1}} \abs{\varphi(y - x) - \varphi(y + x)} \slot \d{y} \\
			& \lesssim \norm{\varphi}_{\holderspace{0, \alpha}}^\eta \int_{-\infty}^{0} \bigabs{\abs{y + h}^{s-1} - \abs{y - h}^{s-1}} \min(x, \abs{y})^{\alpha \eta} \max(x, \abs{y})^{s(1 - \eta)} \slot \d{y} \\
			& = \norm{\varphi}_{\holderspace{0, \alpha}}^\eta x^{\alpha \eta} \int_{-\infty}^{-x} \bigabs{\abs{y + h}^{s-1} - \abs{y - h}^{s-1}} \abs{y}^{s(1 - \eta)} \slot \d{y} \\
			&\eqindent + \norm{\varphi}_{\holderspace{0, \alpha}}^\eta x^{s(1 - \eta)} \int_{-x}^{0} \bigabs{\abs{y + h}^{s-1} - \abs{y - h}^{s-1}} \abs{y}^{\alpha \eta} \slot \d{y} \\
			& = \norm{\varphi}_{\holderspace{0, \alpha}}^\eta x^{\alpha \eta} h^{s + s(1-\eta)} \int_{-\infty}^{-\frac{x}{h}} \bigabs{\abs{t + 1}^{s-1} - \abs{t - 1}^{s-1}} \abs{t}^{s(1 - \eta)} \slot \d{t} \\
			& \eqindent + \norm{\varphi}_{\holderspace{0, \alpha}}^\eta x^{s(1 - \eta)} h^{s + \alpha \eta} \int_{-\frac{x}{h}}^{0} \bigabs{\abs{t + 1}^{s-1} - \abs{t - 1}^{s-1}} \abs{t}^{\alpha \eta} \slot \d{t}.
		\end{aligned}
	\end{equation}
	For the integrand in the second last line we have the identity
	\begin{equation*}
		\bigabs{\abs{t + 1}^{s-1} - \abs{t - 1}^{s-1}} \lesssim \abs{t}^{s-2}
	\end{equation*}
	for large $t$. Thus, we need to choose $\eta$ such that $s-2+s(1-\eta) < -1$ for convergence. But this is possible for every $s \in (0, 1)$ by requiring $\eta > 2 - \frac{1}{s}$. For the integral in the last line, one can show by the mean value theorem and a direct calculation that
	\begin{equation*}
		\int_{-z}^0 \bigabs{\abs{t + 1}^{s-1} - \abs{t - 1}^{s-1}} \abs{t}^{\alpha \eta} \slot \d{t} \lesssim 1 + z^{s - 1 + \alpha \eta}
	\end{equation*}
	for $z > 0$ whenever $s - 2 + \alpha \eta \neq -1$ (a case which can always be avoided by choosing ${\eta \neq (1 - s) / \alpha}$). Using this with $z = x/h$, we find
	\begin{equation*}
		\begin{aligned}
			& x^{s(1 - \eta)} h^{s + \alpha \eta} \int_{-\frac{x}{h}}^{0} \bigabs{\abs{t + 1}^{s-1} - \abs{t - 1}^{s-1}} \abs{t}^{\alpha \eta} \slot \d{t} \\
			\slot & \lesssim x^{\alpha \eta + s(1 - \eta)} h^s \bigparanth{\frac{h}{x}}^{\alpha \eta} \bigparanth{1 + \bigparanth{\frac{x}{h}}^{s - 1 + \alpha \eta}} = x^{\alpha \eta + s(1 - \eta)} h^s \bigparanth{\bigparanth{\frac{h}{x}}^{\alpha \eta} + \bigparanth{\frac{h}{x}}^{1 - s}} \\
			& \lesssim x^{\alpha \eta + s(1 - \eta)} h^s,
		\end{aligned}
	\end{equation*}
	uniformly for $0 < h < x$. We conclude that
	\begin{equation*}
		C_s\int_{-\infty}^0 \bigabs{\abs{y + h}^{s-1} - \abs{y - h}^{s-1}} \abs{\varphi(y - x) - \varphi(y + x)} \slot \d{y} \lesssim \norm{\varphi}_{\holderspace{0, \alpha}}^\eta x^{\alpha \eta + s(1 - \eta)} h^s
	\end{equation*}
	uniformly for $0 < h < x$, with $\eta > 2 - 1/s$. For the regular part of the kernel, it is enough to use the estimate
	\begin{equation*}
		\abs{\varphi(y - x) - \varphi(y + x)} \lesssim \norm{\varphi}_{\holderspace{0, \alpha}}^\eta \min(x, \abs{y})^{\alpha \eta} \norm{\varphi}_{L^\infty}^{1-\eta} \lesssim \norm{\varphi}_{\holderspace{0, \alpha}}^\eta x^{\alpha \eta}
	\end{equation*}
	instead of \eqref{eq:interpolation}. Indeed, inserting this in \eqref{eq:interpolation_startingpoint} for the regular part gives
	\begin{equation} \label{eq:regular_global_estimate}
		\begin{aligned}
			\int_{-\infty}^0 \abs{\besselkernelreg{s}(y + h) - \besselkernelreg{s}(y - h)} \abs{\varphi(y - x) - \varphi(y + x)} \slot \d{y} \lesssim \norm{\varphi}_{\holderspace{0, \alpha}}^\eta h x^{\alpha \eta} \int_{-\infty}^{0} \remainder{h}{1}(y) \slot \d{y},
		\end{aligned}
	\end{equation}
	where the last integral converges due to \eqref{eq:k_reg_first_derivative}. Observe that
	\begin{equation*}
		h x^{\alpha \eta} = x^{\alpha \eta + s(1-\eta)} \frac{h^{s(1-\eta)}}{x^{s(1-\eta)}} h^{1- s(1-\eta)} < x^{\alpha \eta + s(1-\eta)} h^s
	\end{equation*}
	for $0 < h < x < 1$, since we have made the choice of $\eta > 2 - 1/s$. Thus, combining \eqref{eq:singular_global_estimate} and \eqref{eq:regular_global_estimate} with \eqref{eq:interpolation_startingpoint} yields
	\begin{equation*}
			(\wavespeed - \varphi(x)) \abs{\varphi(x + h) - \varphi(x - h)} \lesssim \norm{\varphi}_{\holderspace{0, \alpha}}^\eta x^{\alpha \eta + s(1-\eta)} h^{s},
	\end{equation*}
	with no hidden dependence of $\alpha$ in the inequality. This means that
	\begin{equation*}
		\biggparanth{\frac{\wavespeed-\varphi(x)}{x^{\alpha \eta + s(1-\eta)}}} \biggparanth{\frac{\abs{\varphi(x+h) - \varphi(x-h)}}{h^s}} \lesssim \norm{\varphi}_{\holderspace{0, \alpha}}^{\eta},
	\end{equation*}
	uniformly for $\alpha \in (0, s)$. Since $\wavespeed - \varphi(x) \gtrsim |x|^s$ for $x \ll 1$ by Proposition~\ref{prop:upper_regularity_bound} and $h < x$, this can be reduced to
	\begin{equation*}
		\frac{\abs{\varphi(x + h) - \varphi(x - h)}}{h^{s - \eta(s - \alpha)}} \lesssim \norm{\varphi}_{\holderspace{0, \alpha}}^{\eta}.
	\end{equation*}
	Splitting the estimate over $\eta$ we arrive at
	\begin{equation*}
		\biggparanth{\frac{\abs{\varphi(x + h) - \varphi(x - h)}}{h^{\alpha}}}^{\eta} \biggparanth{\frac{\abs{\varphi(x + h) - \varphi(x - h)}}{h^{s}}}^{1 - \eta} \lesssim \norm{\varphi}_{\holderspace{0, \alpha}}^{\eta},
	\end{equation*}
	which finally proves
	\begin{equation*}
		\sup_{0<h<x<\delta} \left( \frac{\abs{\varphi(x + h) - \varphi(x - h)}}{h^{\alpha}} \right)^{1 - \eta} \lesssim 1
	\end{equation*}
	uniformly for $\alpha \in (0, s)$ with $\max(0, 2-1/s) < \eta < 1$ fixed. This justifies letting $\alpha \nearrow s$, thereby proving global $s$-Hölder regularity of the solution $\varphi$.
\end{proof}

\subsection{Bifurcation to a highest wave} \label{subsec:bifurcation_fkdv}

For ${\beta \in (s, 1)}$ and $P < \infty$, define the function
\begin{equation*} \label{eq:bifurcation_func_fKdV_def}
	F \colon (\varphi, \wavespeed) \mapsto \wavespeed \varphi - \frac{1}{2} \varphi^2 - \bessel{-s} \varphi
\end{equation*}
mapping $\holderspaceeven{0, \beta}(\ptorus) \times \R$ to $\holderspaceeven{0, \beta}(\ptorus)$. At any point $(\varphi, \wavespeed)$ the Fréchet derivative is given by
\begin{equation} \label{eq:frechet_der_in_u}
	\partial_{\varphi} F[\varphi, \wavespeed] = (\wavespeed - \varphi) \id - \bessel{-s},
\end{equation}
where $\id$ denotes the identity operator. Solutions to the equation
\begin{equation} \label{eq:bifurcation_eq}
	F(\varphi, \wavespeed) = 0
\end{equation}
coincide with solutions to the steady fKdV equation, now with the additional requirement of evenness, $P$-periodicity and $\beta$-Hölder continuity of $\varphi$. There are exactly two curves of constant solutions, namely
\begin{equation*}
	\varphi \equiv 0 \qquad \textnormal{and} \qquad \varphi \equiv 2(\wavespeed - 1).
\end{equation*}

The following proposition is an application of an analytic Crandall--Rabinowitz theorem, giving existence of local bifurcation branches around the trivial solution curve $(0, \wavespeed)$ of \eqref{eq:bifurcation_eq}.

\begin{proposition} \label{prop:local_bifurcation}
    For any period $P < \infty$ and $k \in \N$ there exists $\wavespeed_{P, k}^* = \jap{\frac{2 \pi k}{P}}{-s}$ and a local analytic curve
	\begin{equation*}
		\localbifurcationcurve_{P, k} = \set{(\varphi_{P, k}(t), \wavespeed_{P, k}(t)) \setsep t \in (-\varepsilon, \varepsilon) \textnormal{ and } (\varphi_{P, k}(0), \wavespeed_{P, k}(0)) = (0, \wavespeed_{P, k}^*)}
	\end{equation*}
	in $\holderspaceeven{0, \beta}(\ptorus) \times \R$ that bifurcates from the trivial solution curve of \eqref{eq:bifurcation_eq} in $(0, \wavespeed_{P, k}^*)$, such that $F(\varphi_{P, k}(t), \wavespeed_{P, k}(t)) = 0$ for all $t \in (-\varepsilon, \varepsilon)$.
	
	Together with the transcritical bifurcation of constant solutions ${2(\wavespeed - 1)}$, the curves ${\localbifurcationcurve_{P, k}}$ constitute all nonzero solutions to \eqref{eq:bifurcation_eq} in $\holderspaceeven{0, \beta}(\ptorus) \times \R$ in a neighborhood of the trivial solution curve.
\end{proposition}

\begin{proof}
	We check the assumptions of \cite[Theorem 8.3.1]{buffoni_toland}. The Fréchet derivative of $F$ on the trivial curve is
    \begin{equation*}
        \partial_{\varphi}F[0, \wavespeed] = \wavespeed \id - \bessel{-s}.
	\end{equation*}
	The operator $\bessel{-s}$ is a compact automorphism on $\holderspaceeven{0, \beta}(\ptorus)$ owing to the compact embedding
	\begin{equation} \label{eq:compact_zygmund_embedding}
		\zygmundspace{\beta + s}(\ptorus) \doublehookrightarrow \zygmundspace{\beta}(\ptorus)
	\end{equation}
	for $s > 0$ and any finite $P > 0$ (see e.g \cite[A.39]{taylor_3}). As a result of the Fredholm alternative, this implies that $\partial_{\varphi}F[0, \wavespeed]$ is a Fredholm operator of index zero. Furthermore, $\partial_{\varphi}F[0, \wavespeed_{P, k}^*]$ maps the basis function $\varphi_{P, k}^* = \cos\big( \frac{2 \pi k}{P} x \big)$ of $\holderspaceeven{0, \beta}(\ptorus)$ to zero while all others are multiplied by a positive constant. Hence, the dimension of the kernel and the codimension of the image of $\partial_{\varphi}F[0, \wavespeed_{P, k}^*]$ is $1$. Next, we have
	\begin{equation*}
		\ker(\partial_{\varphi}F[0, \wavespeed_{P, k}^*]) = \set{\tau \varphi_{P, k}^*\setsep \tau \in \R }\qquad \textnormal{and} \qquad \partial_{\varphi \wavespeed}^{2} F[0, \wavespeed_{P, k}^*] (1, \varphi_{P, k}^*) = \varphi_{P, k}^*.
	\end{equation*}
	This means that the transversality condition holds, that is
	\begin{equation*}
		\partial_{\wavespeed \varphi}^{2} F[0, \wavespeed_{P, k}^*] (1, \varphi_{P, k}^*) \not \in \image(\partial_{\varphi}F[0, \wavespeed_{P, k}^*]).
	\end{equation*}
	This shows that the assumptions of \cite[Theorem 8.3.1]{buffoni_toland} are satisfied, and we conclude that local bifurcation occurs and that the curves $\localbifurcationcurve_{P, k}$ are analytic since $F$ is analytic.
	
	Since the kernel of $\partial_{\varphi}F[0, \wavespeed]$ is trivial for all $\wavespeed \neq \wavespeed_{P, k}^*$ with $\wavespeed \neq 1$, it follows from the implicit function theorem that the trivial solution curve is otherwise locally unique.
\end{proof}

Hereafter we consider only the first bifurcation point $(0, \wavespeed_{P, 1}^*)$ and the corresponding one-dimensional basis $\varphi_{P, 1}^* = \cos(\frac{2 \pi}{P} x)$ for $\ker \partial_{\varphi} F[0, \wavespeed_{P, 1}^*]$. To simplify notation, let $(\varphi(t), \wavespeed(t))$ denote the curve $\localbifurcationcurve_{P, 1}$ from Proposition~\ref{prop:local_bifurcation}, emanating from the point $(0, \wavespeed_{P, 1}^*)$.

In the analytic setting, we may expand $(\varphi(t)), \wavespeed(t)$ around $t = 0$ as
\begin{equation} \label{eq:local_branch_parametrization}
	\varphi(t) = \sum_{n = 1}^{\infty} \varphicoeff{n} t^n, \qquad \wavespeed(t) = \sum_{n = 0}^{\infty} \wavespeedcoeff{2n} t^{2n},
\end{equation}
corresponding to the Lyapunov-Schmidt reduction \cite{buffoni_toland}, where we have used $\wavespeed(t) = \wavespeed(-t)$ (see \cite{ehrnstrom_wahlen}). Then $\wavespeedcoeff{0} = \wavespeed_{P, 1}^* = \jap{\frac{2 \pi}{P}}{-s}$ and $\varphicoeff{1}(x) = \cos(\frac{2 \pi}{P} x)$. Furthermore, one can check that
\begin{equation*}
	\varphicoeff{2}(x) = -\frac{1}{4(1 - \jap{\frac{2 \pi}{P}}{-s})} - \frac{1}{4(\jap{\frac{4 \pi}{P}}{-s} - \jap{\frac{2 \pi}{P}}{-s})} \cos\big( \frac{4 \pi}{P} x\big).
\end{equation*}
and
\begin{equation*} \label{eq:second_derivative_mu}
	\wavespeedcoeff{2} = \frac{1}{4(\jap{\frac{2 \pi}{P}}{-s} - 1)} + \frac{1}{8(\jap{\frac{2 \pi}{P}}{-s} - \jap{\frac{4 \pi}{P}}{-s})}.
\end{equation*}

In the direction of global bifurcation, we define the sets
\begin{equation*}
	\submaxset = \bigset{(\varphi, \wavespeed) \in \holderspaceeven{0, \beta}(\ptorus) \times \R \setsep \varphi < \wavespeed}\qquad \textnormal{and} \qquad \solutionset = \bigset{ (\varphi, \wavespeed) \in U \setsep F(\varphi, \wavespeed) = 0},
\end{equation*}
and let $\solutionset^1$ denote the $\varphi$-component of $\solutionset$.

\begin{proposition} \label{prop:global_bifurcation}
	The local bifurcation branch $t \mapsto (\varphi(t), \wavespeed(t))$ extends to a global continuous curve $\globalbifurcationcurve = \set{(\varphi(t), \wavespeed(t)) \setsep t \in [0, \infty)} \subset \submaxset$, and one of the following alternatives holds.
	\begin{itemize}
		\item[(i)] $\norm{(\varphi(t), \wavespeed(t))}_{\holderspace{0, \beta} \times \R} \rightarrow \infty$ as $t \rightarrow \infty$,
		\item[(ii)] $\dist(\globalbifurcationcurve, \boundary U) = 0$,
		\item[(iii)] $\globalbifurcationcurve$ is a closed loop of finite period. That is, there exists $T > 0$ such that
		\begin{equation*}
			\globalbifurcationcurve = \set{(\varphi(t), \wavespeed(t)) \setsep 0 \leq t \leq T},
		\end{equation*}
		where $(\varphi(T), \wavespeed(T)) = (0, \wavespeed_{P, 1}^*)$.
	\end{itemize}
\end{proposition}

\begin{proof}
	We check the assumptions of \cite[Theorem 9.1.1]{buffoni_toland} (see also \cite[Theorem 6]{constantin_strauss} for comments on the condition $\dot{\wavespeed} \not \equiv 0$ around $t = 0$ which we do not use here). Firstly, note that the operator $\partial_{\varphi}F[\varphi, \wavespeed]$ given in \eqref{eq:frechet_der_in_u} is Fredholm of index zero for every $(\varphi, \wavespeed) \in \submaxset$. Indeed, $(\wavespeed - \varphi) \id$ is a linear homeomorphism on $\holderspaceeven{0, \beta}(\ptorus)$ for $\varphi < \wavespeed$, and $\bessel{-s}$ is compact, so the claim follows from \cite[Theorem 2.7.6]{buffoni_toland}.
	
	Secondly, every closed and bounded subset of $\solutionset$ is compact: if $K$ is a closed and bounded subset of $\solutionset$, then $K^1 = \set{\varphi \setsep (\varphi, \wavespeed) \in K}$ is a bounded subset of $\zygmundspaceeven{\beta + s}(\ptorus)$ due to \eqref{eq:bootstrap_form}. In view of the compact embedding \eqref{eq:compact_zygmund_embedding} we see that $K^1$ is relatively compact in $\holderspaceeven{0, \beta}(\ptorus)$. But $K$ is closed by assumption, so it is compact. Since we have already proved the existence of local bifurcation in Proposition~\ref{prop:local_bifurcation}, we are done.
\end{proof}

Towards invoking \cite[Theorem 9.2.2]{buffoni_toland} and the exclusion of alternative (iii), let the closed cone $\cone$ be defined as
\begin{equation} \label{eq:definition_of_cone}
	\cone = \set{ \varphi \in \holderspaceeven{0, \beta}(\ptorus) \setsep \varphi \textnormal{ is nondecreasing on } (\nicefrac{-P}{2}, 0)}.
\end{equation}
Then we have the following.

\begin{proposition} \label{prop:exclusion_of_alt_3}
    The first component $\varphi(t)$ of the global bifurcation curve $\globalbifurcationcurve$ belongs to $\cone \setminus \{0\}$ for all $t > 0$, and alternative (iii) in Proposition~\ref{prop:global_bifurcation} does not occur.
\end{proposition}

\begin{proof}
	We claim that in $\solutionset^1$, every nonconstant function $\varphi$ in $\globalbifurcationcurve^1 \cap \cone$ lies in the interior of $\cone$. Such a solution $\varphi$ must be smooth with $\varphi' > 0$ on $(\nicefrac{-P}{2}, 0)$ and furthermore $\varphi''(0) < 0$ and $\varphi''(-P/2) > 0$, by Lemma~\ref{lemma:regularity_1} and Lemma~\ref{lemma:nodal_properties}. Now if $\psi \in \solutionset^1$ with $\norm{\varphi - \psi}_{C^{0, \beta}} < \delta$ for some $\delta > 0$, then $\norm{\varphi - \psi}_{C^2} < \tilde{\delta}$ where $\tilde{\delta}$ can be made arbitrarily small at the expense of $\delta$. This means that $\psi' > 0$ on some closed subset $[a, b]$ of $(\nicefrac{-P}{2}, 0)$. Suppose that $\psi' \leq 0$ on $(b, 0)$. Then $\psi'(0) < \psi'(x) \leq 0$ for $x \in (b, 0)$, since $\psi''(0) < 0$. But this contradicts the evenness of $\psi$. With an analogous argument on $(\nicefrac{-P}{2}, a)$, we arrive at $\psi' \geq 0$ on $(\nicefrac{-P}{2}, 0)$. Thus, $\psi \in \cone$, and $\varphi$ belongs to the interior of $\cone$. Together with Lemma~\ref{lemma:a_priori_l2_norm} this suffices to exclude the alternative (iii) (see also \cite[Theorem 6.7]{ehrnstrom_wahlen} for a more detailed explanation why the transcritical bifurcation ${2(\wavespeed - 1)}$ does not cause problems here).
\end{proof}

\begin{lemma} \label{lemma:conv_subsequence_global_bifurcation}
	Any sequence of solutions $(\varphi_n, \wavespeed_n)_{n \in \N} \subset \solutionset$ to the steady fKdV equation with bounded $(\wavespeed_n)_{n \in \N}$ converges uniformly along a subsequence to a solution $(\varphi, \wavespeed)$.
\end{lemma}

\begin{proof}
	It follows directly from the equation that
	\begin{equation*}
		\norm{\varphi}_{\lebesguespace{\infty}(\R)}^2 \leq 2 \norm{\wavespeed \varphi}_{\lebesguespace{\infty}(\R)} + 2 \norm{\bessel{-s}\varphi}_{\lebesguespace{\infty}(\R)} \leq 2(\abs{\wavespeed} + 1) \norm{\varphi}_{\lebesguespace{\infty}(\R)}
	\end{equation*}
	so $(\varphi_n)_n$ is bounded provided $(\wavespeed_n)_n$ is bounded. Furthermore, the sequence $(\bessel{-s} \varphi_n)_n$ is uniformly equicontinuous since $\besselkernel{s}$ is integrable and continuous outside of zero. Then due to Arzela--Ascoli, the sequence $(\bessel{-s}\varphi_n)_n$ has a uniformly convergent subsequence. Uniform convergence ensures that the limit is also a solution of the equation.
\end{proof}

\newcommand{\bifurcationright}[3]{%
\draw[blue, thick, densely dashed] plot[smooth, tension=1]
coordinates{(#1, #2) (#1 + 0.01, #2 + 0.125 * #3) (#1 + 0.11, #2 + 0.6 * #3) (#1, #2 + #3)};
}
\newcommand{\bifurcationleft}[3]{%
\draw[blue, thick] plot[smooth, tension=1]
coordinates{(#1, #2) (#1-0.01, #2 + 0.125 * #3) (#1 - 0.11, #2 + 0.6 * #3) (#1, #2 + #3)};
}
\newcommand{\symmetrytransformation}[4]{%
\draw[<->, dashed, thick] (#1 - 0.02, #2 - 0.22) to[out=-90, in=180] (3, 1.1) to[out=0, in=-90] (#3 + 0.02, #4 - 0.22);
}

\begin{figure}[t]
\centering
    \begin{subfigure}[b]{0.6\textwidth}
        \centering
        \resizebox{0.9\textwidth}{!}{%
        \begin{tikzpicture}
            \draw[line width=0.5em, gray!50, cap=round]
            (0.1, 0.1) to (2.9, 2.9);
            \draw[line width=0.5em, gray!50, cap=round]
            (3.14, 3) to (6.3, 3);
            \draw[thick, blue, ->]
            (-0.3, 3) to (6.5,3);
            \draw[black]
            (6.5, 3) coordinate[label = {below:$\mu$}];
            \draw[thick, ->]
            (0,0) to (0,6) coordinate[label = {above:$\max \varphi$}] (ymax);
            \draw[thick, blue]
            (0, 0) to (6, 6);
            \draw[thick, black]
            (0, 3) to (6, 6);
            \draw[dotted, thick]
            (3, 1) coordinate[label = {below:$\mu \mapsto 2 - \mu$}] to (3, 6.2) coordinate[label = {above:$\mu = 1$}];
            \draw[dotted, thick]
            (6, 1) to (6, 6.2) coordinate[label = {above:$\mu = 2$}];
            \bifurcationleft{2.4}{3}{1.2}
            \bifurcationleft{1.7}{3}{0.85}
            \bifurcationleft{1}{3}{0.5}
            \bifurcationright{3.6}{3.6}{1.2}
            \bifurcationright{4.3}{4.3}{0.85}
            \bifurcationright{5}{5}{0.5}
            \symmetrytransformation{2.4}{3}{3.6}{3.6}
            \draw
            (5.1, 5.8) coordinate[label = {[rotate=28]left:$\max \varphi = \mu$}];
            \draw
            (1.75, 2.3) coordinate[label = {[rotate=45]left:$\varphi \equiv 2(\mu - 1)$}];
        \end{tikzpicture}
        }
        \vspace{0.9cm}
        \caption{Bifurcation diagram}
        \label{fig:bifurcation_sub_diagram}
    \end{subfigure}
    \begin{subfigure}[b]{0.35\textwidth}
        \begin{subfigure}[b]{\textwidth}
            \centering
            \includegraphics[width=\textwidth]{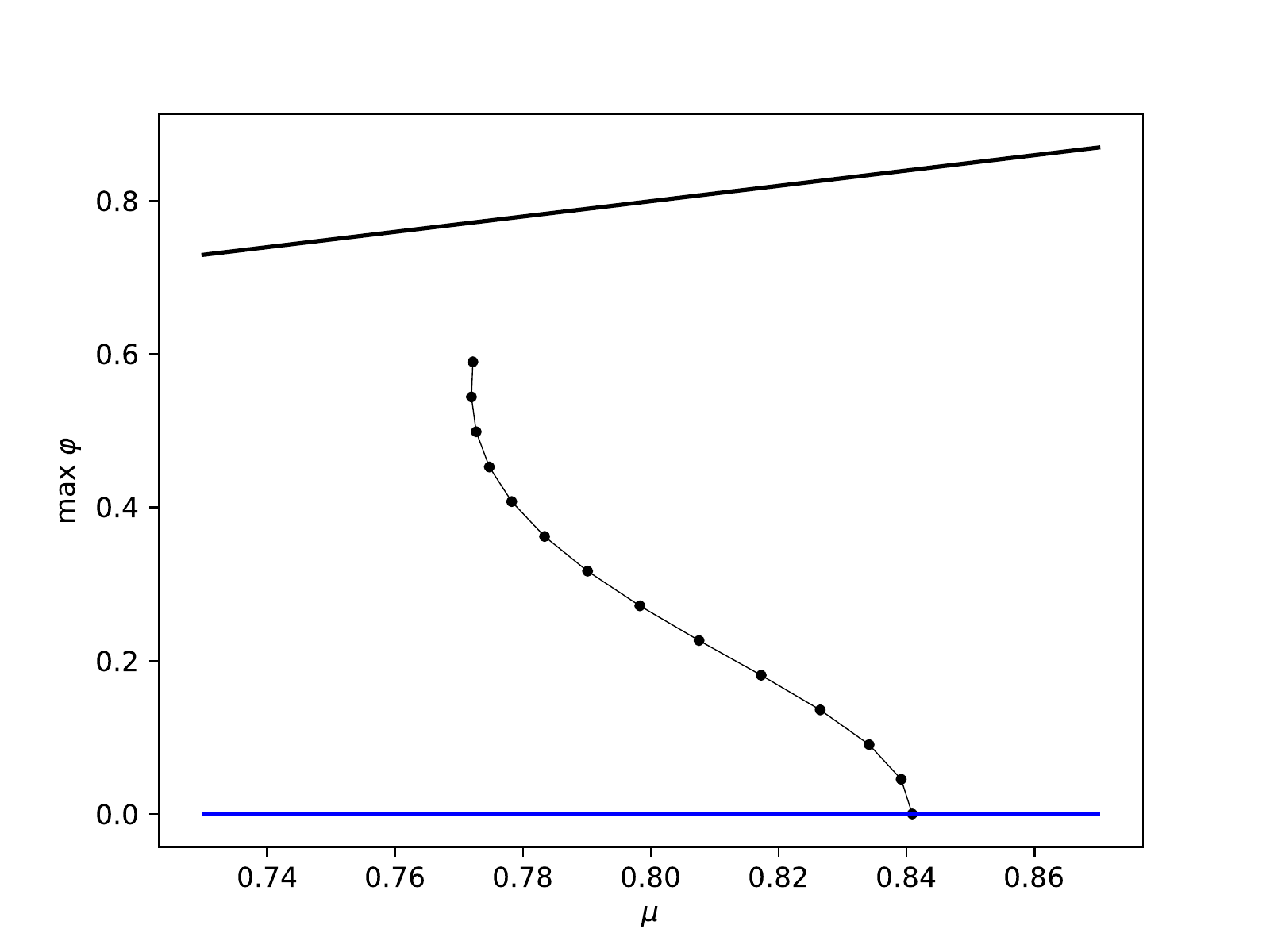}
            \caption{Numerical bifurcation branch}
            \label{fig:branch_fkdv}
        \end{subfigure}

        \begin{subfigure}[b]{\textwidth}
            \centering
            \includegraphics[width=\textwidth]{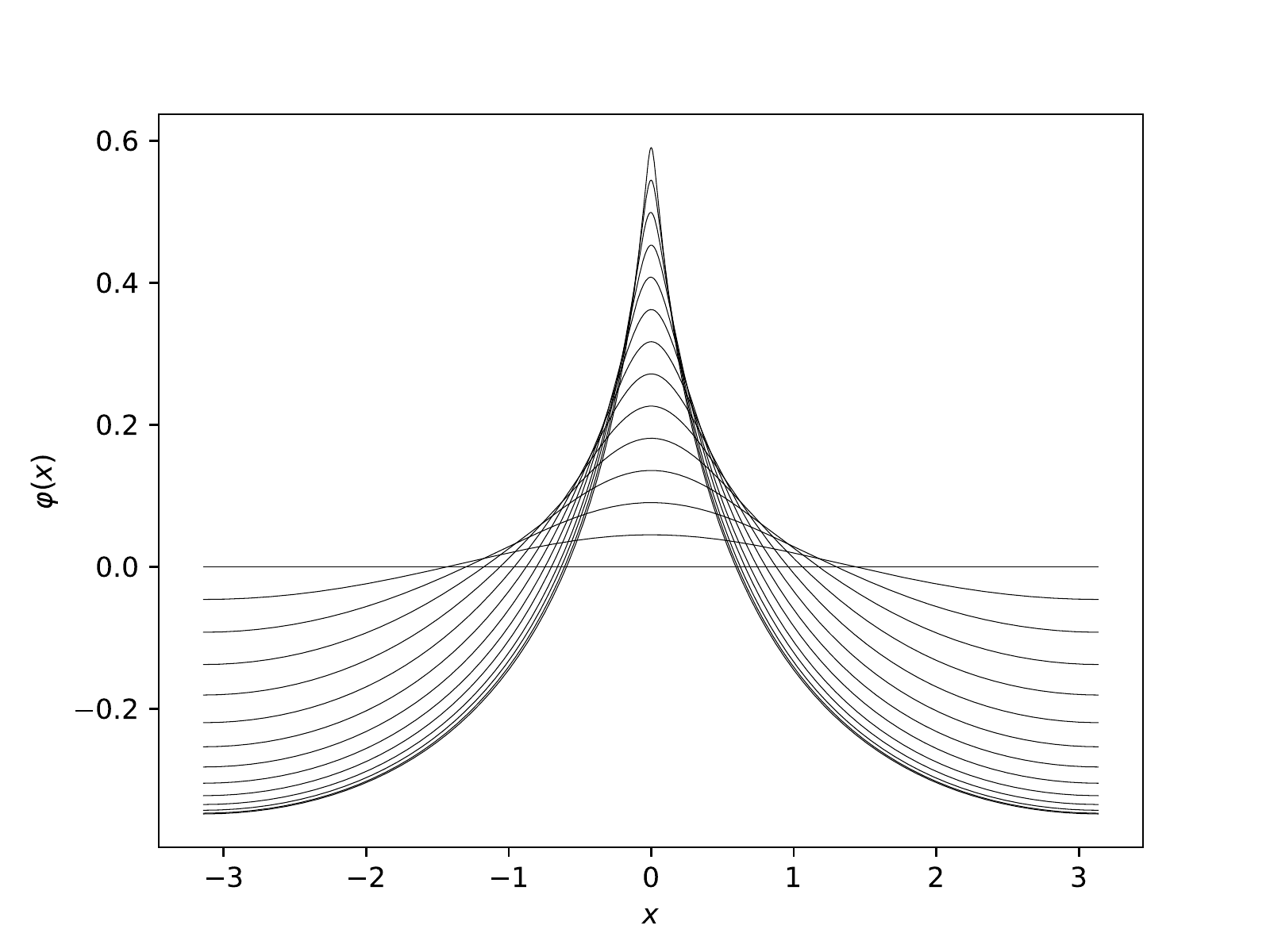}
            \caption{Traveling waves along numerical branch}
            \label{fig:waves_fkdv}
        \end{subfigure}
    \end{subfigure}

    \caption{(a) Bifurcation branches emanating from the trivial solution curve for $\wavespeed \in (0, 1)$, reflected to $\wavespeed \in (0, 1)$ via the Galilean transformation ${(\varphi, \wavespeed) \mapsto (\varphi + 2(1 - \wavespeed), 2 - \wavespeed)}$. Otherwise the lines consisting of constant solutions are locally unique. Local branches extend to global curves, and a highest, cusped traveling wave with $\varphi(0) = \wavespeed$ can be found in the limit. \mbox{(b)-(c)} Numerical example with $P = 2\pi$ and $s = 0.5$.}
    \label{fig:bifurcation_diagram}
\end{figure}

We are now in the position to conclude that a highest traveling-wave solution to the steady fKdV equation exists at the limit of the bifurcation curve $\globalbifurcationcurve$.

\begin{proof}[Proof of Theorem \ref{thm:alt_1_and_2_occur}]
	First we claim that the second component $\wavespeed(t)$ of $\globalbifurcationcurve$ is strictly bounded between $0$ and $1$. Indeed, if $\varphi(t)$ is to cross the line $\wavespeed = 1$, then it would have to vanish by Lemma~\ref{lemma:a_priori_l2_norm}, contradicting Proposition~\ref{prop:exclusion_of_alt_3}. On the other hand, assume that there is a sequence $(\wavespeed_n)_n$ with $\wavespeed_n \rightarrow 0$ as $n \rightarrow 0$. Then by Lemma~\ref{lemma:conv_subsequence_global_bifurcation} there is a uniformly convergent subsequence of $(\varphi_n)_n$, converging to some $\varphi_0$, which is also a solution to the steady fKdV equation. But since $\varphi_n < \wavespeed_n$ along the bifurcation branch, taking the limit one obtains $\varphi_0 \leq 0$. This means that $ \max_{x} \varphi_0(x) = 0$ by Proposition~\ref{prop:a_priori_inf_sup} and therefore $\varphi_0 \equiv 0$. But then
	\begin{equation*}
		0 = \lim_{n \rightarrow \infty} (\wavespeed_n - \varphi_n(P/2)) \gtrsim 1
	\end{equation*}
	owing to Proposition~\ref{prop:upper_regularity_bound}: a contradiction.

	Next, we show that alternative (i) and (ii) from Proposition~\ref{prop:global_bifurcation} occur simultaneously. Assume first that (i) occurs when $t \rightarrow \infty$. This can only happen if $\norm{\varphi(t)}_{\holderspace{0, \beta}} \rightarrow \infty$ since $\wavespeed$ is bounded from above. Aiming at a contradiction, suppose that there exists $\delta > 0$ with
	\begin{equation*}
		\liminf_{t \rightarrow \infty} \inf_{x \in \R} (\wavespeed(t) - \varphi(t)(x)) \geq \delta.
	\end{equation*}
	Then using \eqref{eq:point_comparison_form}, we have for every $x,y \in \R$ that
	\begin{equation*}
		\abs{\varphi(x) - \varphi(y)} = \frac{2\abs{(\bessel{-s}\varphi)(x) - (\bessel{-s}\varphi)(y)}}{\abs{2\wavespeed - \varphi(x) - \varphi(y)}} \leq \frac{|(\bessel{-s}\varphi)(x) - (\bessel{-s}\varphi)(y)|}{\delta}.
	\end{equation*}
	Starting with bounded $\varphi$, iteration of $\bessel{-s}\colon L^{\infty} \rightarrow \zygmundspace{s}$ and $\bessel{-s} \colon \zygmundspace{\beta} \rightarrow \zygmundspace{\beta + s}$ yields $\varphi \in \holderspace{0, \alpha}$ for some $\alpha > \beta$. But now $\norm{\varphi(t)}_{\holderspace{0, \beta}}$ is bounded, which is a contradiction.

	Conversely, suppose (ii) occurs. That is, there exists a sequence $(\varphi_n, \wavespeed_n)_{n \in \N}$ with $\varphi_n' \geq 0$ on $(\nicefrac{-P}{2}, 0)$ and $\varphi_n < \wavespeed_n$ for all $n \in \N$, and
	\begin{equation*}
		\lim_{n \rightarrow \infty} \abs{\wavespeed_n - \varphi_n(0)} = 0.
	\end{equation*}
	Suppose that $\varphi_n$ remains bounded in $\holderspace{0, \beta}(\R)$. Taking the limit of a subsequence in $C^{0, \beta'}(\R)$ for $s < \beta' < \beta$, the limit must be exactly $s$-Hölder regular at the crest by \eqref{eq:exact_sigma_holder_at_crest}, and we arrive at a contradiction to the boundedness of the sequence in $\holderspace{0, \beta}(\R)$. Thus, both alternative (i) and (ii) in Proposition~\ref{prop:global_bifurcation} occur, and for every unbounded sequence $(t_n)_{n \in \mathbb{N}}$ of positive numbers, there exists a subsequence of $(\varphi(t_n), \wavespeed(t_n))_{n \in \mathbb{N}}$ that converges to a solution $(\varphi, \wavespeed)$ to the steady fKdV equation, with $\varphi(0) = \wavespeed$. The limiting wave is even, $P$-periodic, strictly increasing on $(\nicefrac{-P}{2}, 0)$ and exactly \mbox{$s$-Hölder} continuous at $x \in P \mathbb{Z}$.
\end{proof}

\section{The fDP equation} \label{sec:fdp}

We now turn our attention to the steady fDP equation \eqref{eq:main_steady_fDP} with parameters $s \in (0, 1)$ and $\intconst \in \R$ fixed. In Section~\ref{subsec:traveling_waves_regularity_fdp} we prove a priori results about magnitude and regularity of solutions. Existence is then derived by means of a bifurcation argument in Section~\ref{subsec:bifurcation_fdp}. We mainly follow the framework which was used for the fKdV equation above. Inspiration has also been taken from \cite{arnesen}. Details are sometimes omitted to avoid unnecessary repetition.

\subsection{Periodic traveling waves and regularity} \label{subsec:traveling_waves_regularity_fdp}

There are two constant solutions to the steady fDP equation, given by
\begin{equation*} \label{eq:constant_solutions_fDP}
	\constsolneg = \frac{\wavespeed - \sqrt{\wavespeed^2 + 8\intconst}}{4} \qquad \textnormal{and} \qquad \constsolpos = \frac{\wavespeed + \sqrt{\wavespeed^2 + 8\intconst}}{4}.
\end{equation*}
Note that if $\varphi$ solves the steady fDP equation with wave-speed $\wavespeed$, then $-\varphi(-x)$ is a solution to the equation with $-\wavespeed$. So we assume from now on that $\wavespeed > 0$. Writing the fDP equation in the form \eqref{eq:main_steady_fdp_linearized} and using that $\altoperator{-s} \varphi \geq \min \varphi$ and $\altoperator{-s} \varphi \leq \max \varphi$ with strict inequality for nonconstant solutions, we obtain the following.

\begin{proposition} \label{prop:a_priori_inf_sup_fDP}
	If $\varphi$ is a solution to the steady fDP equation, then
	\begin{equation*} \label{eq:solution_ranges_fdp}
		\constsolneg \leq \min \varphi \leq \constsolpos \leq \max \varphi\ \ \textnormal{ or }\ \ \varphi \equiv \constsolneg.
	\end{equation*}
\end{proposition}

The previous proposition shows that if $\intconst \leq 0$ then all solutions are nonnegative, and for $\intconst < -\frac{\wavespeed^2}{8}$ there are no real solutions. We now prove a lemma concerning the nodal properties of solutions to the steady fDP equation.

\begin{lemma} \label{lemma:nodal_properties_fdp}
    Let $P < \infty$. Every $P$-periodic, nonconstant and even solution $\varphi \in \contderspace{1}(\R)$ to the steady fDP equation which is nondecreasing on $(\nicefrac{-P}{2}, 0)$ satisfies
    \begin{equation*}
        \varphi' > 0 \qquad \textnormal{and} \qquad \varphi < \wavespeed
    \end{equation*}
    on $(\nicefrac{-P}{2}, 0)$. If in addition $\varphi \in \contderspace{2}(\R)$, then $\varphi''(0) < 0$ and $\varphi''(\pm P/2) > 0$.
\end{lemma}

\begin{proof}
	Under the assumptions above, $\varphi'$ is odd, nontrivial and nonnegative on $(\nicefrac{-P}{2}, 0)$. Differentiation of the equation in the form \eqref{eq:main_steady_fdp_linearized} leads to
    \begin{equation*}
        (\wavespeed - \varphi) \varphi' = \frac{3}{4} \wavespeed \altoperator{-s} \varphi' > 0
    \end{equation*}
    on $(\nicefrac{-P}{2}, 0)$. Hence, $\varphi' > 0$ and $\varphi < \wavespeed$ on $(\nicefrac{-P}{2}, 0)$. Differentiating twice yields
    \begin{equation} \label{eq:twice_differentiated_linearized_form}
        (\wavespeed - \varphi) \varphi'' = \frac{3}{4} \wavespeed \altoperator{-s}\varphi'' + (\varphi')^2,
    \end{equation}
    and proceeding as in the proof of Lemma~\ref{lemma:nodal_properties} by evaluating \eqref{eq:twice_differentiated_linearized_form} in $x = 0$ and using integration by parts and the characterization of $\paltkernel{s}$ from Lemma~\ref{lemma:asymptotic_altkernel_behavior}, we obtain $\varphi''(0) < 0$ and $\varphi''(\pm P/2) > 0$.
\end{proof}

We obtain an analogue of Lemma~\ref{lemma:smooth_away_from_crest} by taking the difference of the steady fDP equation on the form \eqref{eq:main_steady_fdp_linearized} evaluated in two points,
\begin{equation*} \label{eq:point_comparison_form_fDP}
		(2 \wavespeed - \varphi(x) - \varphi(y))(\varphi(x) - \varphi(y)) = \frac{3}{2} \wavespeed \bigparanth{(\altoperator{-s} \varphi)(x) - (\altoperator{-s} \varphi)(y)},
\end{equation*}
and proceeding as in \cite[Lemma 5.2]{ehrnstrom_wahlen}.

\begin{lemma} \label{lemma:smooth_away_from_crest_fDP}
	Let $P \in (0, \infty]$. Assume that $\varphi$ is an even, $P$-periodic and nonconstant solution to the steady fDP equation which is nondecreasing on $(\nicefrac{-P}{2}, 0)$ and with $\varphi \leq \wavespeed$. Then $\varphi$ is strictly increasing on $(\nicefrac{-P}{2}, 0)$.
\end{lemma}

We also have an analogue of Lemma \ref{lemma:regularity_1} by bootstrapping via
\begin{equation} \label{eq:bootstrap_form_fDP}
	\varphi = \wavespeed - \sqrt{\wavespeed^2 + 2 \intconst - 3\bessel{-s} \varphi^2}.
\end{equation}
in the scale of Hölder-Zygmund spaces:

\begin{lemma} \label{lemma:regularity_1_fDP}
    Assume that $\varphi \leq \wavespeed$ is a solution to the steady fDP equation. Then $\varphi$ is smooth on every open set where $\varphi < \wavespeed$.
\end{lemma}

It is clear that traveling waves for the fDP and the fKdV equation share many features. Solutions which are strictly smaller than the wave speed $\wavespeed$ are smooth, but smoothness may break down when the amplitude approaches $\wavespeed$.

\begin{proposition} \label{prop:upper_regularity_bound_fDP}
	Let $P \in (0, \infty]$. Assume that $\varphi$ is an even, $P$-periodic and nonconstant solution to the steady fDP equation which is nondecreasing on $(\nicefrac{-P}{2}, 0)$ with $\varphi \leq \wavespeed$. Then
	\begin{equation*} \label{eq:lower_bound_crest_fDP}
		\wavespeed - \varphi(x) \gtrsim \wavespeed |x|^{s}
	\end{equation*}
	uniformly for $|x| \ll 1$. Moreover, if $P < \infty$ then
	\begin{equation} \label{eq:lower_bound_end_interval_fDP}
		\wavespeed - \varphi(x) \gtrsim \wavespeed.
	\end{equation}
\end{proposition}

\begin{proof}
	We work with the steady fDP equation in the form \eqref{eq:main_steady_fdp_linearized}, assuming first that $P < \infty$. In the same way as in \eqref{eq:lower_bound_starting_point}, we find
	\begin{equation} \label{eq:lower_bound_starting_point_fDP}
		(\wavespeed - \varphi(z)) \varphi'(x) \geq \frac{3}{4} \wavespeed \int_{x_0/2}^{x_0/4} \bigparanth{\paltkernel{s}(x-y) - \paltkernel{s}(x+y)} \varphi'(y) \slot \d{y}
	\end{equation}
	for $x_0 \in (\nicefrac{-P}{2}, 0)$, $x \in (\frac{x_0}{2}, \frac{x_0}{4})$ and $z \in [\nicefrac{-P}{2}, x]$. With
	\begin{equation*}
		C_{P} = \min \set{\paltkernel{s}(x - y) - \paltkernel{s}(x + y) \setsep x, y \in (\frac{x_0}{2}, \frac{x_0}{4})} > 0
	\end{equation*} 
	we deduce 
	\begin{equation*}
		(\wavespeed - \varphi(z)) \geq \frac{3}{16} C_{P} \wavespeed \abs{x_0}.
	\end{equation*}
	Choosing $x_0 = -P/4$ and $z \in (\nicefrac{-P}{2}, \nicefrac{-P}{8})$ gives \eqref{eq:lower_bound_end_interval_fDP}. Next, it suffices to observe that
	\begin{equation*}
		\paltkernel{s}(x - y) - \paltkernel{s}(x + y) \geq -2y \paltkernel{s}'\bigparanth{x_0} \gtrsim_{P} \abs{x_0}^{s-1}
	\end{equation*}
	by the mean value theorem and \eqref{eq:periodic_alt_kernel_singularity} uniformly over $x, y \in (x_0/2, x_0/4)$ with $\abs{x_0} \ll 1$. We insert this is \eqref{eq:lower_bound_starting_point_fDP}, whereupon integration over $x$ and setting $x = x_0$ gives
	\begin{equation*}
		(\wavespeed - \varphi(x_0) \gtrsim \wavespeed (x_0/4 - x_0/2) |x_0|^{s - 1} \gtrsim \wavespeed |x_0|^s
	\end{equation*}
	uniformly for $|x_0| \ll 1$. As before, the estimate can be obtained uniformly for large $P$, thereby proving the solitary case $P = \infty$.
\end{proof}

\begin{theorem} \label{thm:regularity_2_fDP}
	Let $P \in (0, \infty]$, and let $\varphi \leq \wavespeed$ be an even and nonconstant solution to the steady fDP equation which is nondecreasing on $(\nicefrac{-P}{2}, 0)$ and with $\varphi(0) = \wavespeed$. Then $\varphi \in C^{0, s}(\R)$. Moreover,
    \begin{equation*} \label{eq:exact_sigma_holder_at_crest_fDP}
        \wavespeed - \varphi(x) \eqsim |x|^s
    \end{equation*}
    uniformly for $|x| \ll 1$.
\end{theorem}

\begin{proof}
	Since $\varphi$ touches the value $\wavespeed$ in the origin we have
	\begin{equation*}
        \begin{aligned}
            (\wavespeed - \varphi(x))^2 & = 3 (\bessel{-s}\varphi^2)(0) - 3 (\bessel{-s}\varphi^2)(x) \\
            & = \frac{3}{2} \int_{\R} (\besselkernel{s}(x + y) + \besselkernel{s}(x - y) - 2 \besselkernel{s}(y)) (\varphi^2(0) - \varphi^2(y)) \slot \d{y}.
        \end{aligned}
    \end{equation*}
	In addition, one has the formula
	\begin{equation} \label{eq:global_bootstrap_fDP}
		\begin{aligned}
			& (\varphi(x + h) - \varphi(x - h))^2 \\
			& \leq \bigabs{(2\wavespeed - \varphi(x + h) - \varphi(x - h))(\varphi(x + h) - \varphi(x - h))} \\
			& = 3 \bigabs{(\bessel{-s}\varphi^2)(x + h) - (\bessel{-s}\varphi^2)(x - h)}
		\end{aligned}
	\end{equation}
	with
	\begin{equation*}
        (\bessel{-s}\varphi^2)(x+h) - (\bessel{-s} \varphi^2)(x-h) = \int_{-\infty}^0 (\besselkernel{s}(y+h) - \besselkernel{s}(y-h))(\varphi^2(y-x) - \varphi^2(y+x)) \slot \d{y},
    \end{equation*}
	Thus, the simple observation
    \begin{equation*} \label{eq:fDP_to_fKdV_estimate}
        \abs{\varphi^2(x) - \varphi^2(y)} \leq 2 \norm{\varphi}_{\lebesguespace{\infty}} \abs{\varphi(x) - \varphi(y)}
    \end{equation*}
    combined with Proposition~\ref{prop:upper_regularity_bound_fDP} allows us to prove Theorem~\ref{thm:regularity_2_fDP} in the same way as the proof of Theorem~\ref{thm:regularity_2}.
\end{proof}

\subsection{Bifurcation to a highest wave} \label{subsec:bifurcation_fdp}

We set $\beta \in (s, 1)$ and define
\begin{equation*} \label{eq:bifurcation_function_fDP}
	\bifurcationfuncfDP \colon (\varphi, \wavespeed) \mapsto \wavespeed \varphi - \frac{1}{2} \varphi^2 - \frac{3}{2} \bessel{-s} \varphi^2 + \intconst,
\end{equation*}
mapping $\holderspaceeven{0, \beta}(\ptorus) \times \R$ to $\holderspaceeven{0, \beta}(\ptorus)$. It is practical to bifurcate from a line of trivial solutions, so we consider the function
\begin{equation*} \label{eq:perturbed_bifurcation_function_fDP}
	\trivialfuncfDP(\shiftedsol, \wavespeed) = \bifurcationfuncfDP(\constsolpos(\wavespeed) + \shiftedsol, \wavespeed) = (\wavespeed - \constsolpos(\wavespeed))\shiftedsol - \frac{1}{2} \shiftedsol^2 - \frac{3}{2} \bessel{-s}\shiftedsol^2 - 3\constsolpos(\wavespeed)\bessel{-s}\shiftedsol,
\end{equation*}
where $\constsolpos(\wavespeed)$ is the largest constant solution to the steady fDP equation. Nonconstant periodic solutions have to cross this branch of constant solutions as shown in Lemma~\ref{prop:a_priori_inf_sup_fDP}. Moreover, the Fréchet derivative of $\trivialfuncfDP$ with respect to $\phi$ is
\begin{equation} \label{eq:Frechet_fDP_bifurcation}
    \partial_\phi \trivialfuncfDP[0, \wavespeed] = (\wavespeed - \constsolpos(\wavespeed)) \id - 3 \constsolpos(\wavespeed) \bessel{-s},
\end{equation}
and in order to have bifurcation points along the trivial solution curve of $\trivialfuncfDP = 0$ the kernel of $\partial_\phi \trivialfuncfDP[0, \wavespeed]$ must be nontrivial. That is, there must exist $k \in \N$ such that
\begin{equation} \label{fDP_kernel_dimension}
    \bigjap{\frac{2\pi k}{P}}{-s} = \frac{1}{3} \frac{\wavespeed - \constsolpos(\wavespeed)}{\constsolpos(\wavespeed)},
\end{equation}
which is possible only for $\constsolpos$. Constant $\varphi$-solutions of the problem $\bifurcationfuncfDP(\varphi, \wavespeed) = 0$ maps one-to-one to trivial $\phi$-solutions of the problem ${\trivialfuncfDP(\phi, \wavespeed) = 0}$ via the relation $\shiftedsol = \varphi - \constsolpos(\wavespeed)$. This allows us to prove the following lemma.

\begin{proposition} \label{prop:local_bifurcation_fDP}
    Assume that $-\frac{\wavespeed^2}{8} < \intconst < \infty$ and $P < \infty$.
    \begin{itemize}
        \item[(i)] If $\intconst < 0$, then for every $k \in \N$ with $\frac{2 \pi k}{P} < \sqrt{3^{2/s} - 1}$ there exists $\bifurcationpoint \in (\sqrt{-8 \intconst}, \infty)$,
        \item[(ii)] if $\intconst > 0$, then for every $k \in \N$ with $\frac{2 \pi k}{P} > \sqrt{3^{2/s} - 1}$ there exists $\bifurcationpoint \in (\sqrt{\intconst}, \infty)$
    \end{itemize}
    such that $(\constsolpos(\bifurcationpoint), \bifurcationpoint)$ is a bifurcation point for $\bifurcationfuncfDP$ in each case. Around each bifurcation point there is a local analytic curve
	\begin{equation*}
		\localbifurcationcurvefDP_{P, k} = \set{(\varphi_{P, k}(t), \wavespeed_{P, k}(t)) \setsep t \in (-\varepsilon, \varepsilon)} \subset \holderspaceeven{0, \beta}(\ptorus) \times \R
	\end{equation*}
    such that $\bifurcationfuncfDP(\varphi_{P, k}(t), \wavespeed_{P, k}(t)) = 0$ for all $t \in (-\varepsilon, \varepsilon)$ and $\varphi_{P, k}(0) = \constsolpos(\bifurcationpoint)$. Furthermore, the curves $\localbifurcationcurvefDP_{P, k}$ constitute all nonconstant solutions of the steady fDP equation in a neighborhood of the two constant solution curves.
\end{proposition}

\begin{proof}
    Since solutions of the problem $G(\varphi, \wavespeed) = 0$ map one-to-one to solutions of ${\trivialfuncfDP(\phi, \wavespeed) = 0}$, it suffices to establish the existence of local bifurcation curves $\localshiftedcurvefDP_{P, k}$ of $\trivialfuncfDP$. We check the assumptions of Crandall-Rabinowitz \cite[Theorem 8.3.1]{buffoni_toland}. The Fréchet derivative of $\trivialfuncfDP$ given by \eqref{eq:Frechet_fDP_bifurcation} is a sum of the scaled identity and the scaled compact operator $\bessel{-s}$. As we have seen before, this implies that $\partial_\phi \trivialfuncfDP[0, \wavespeed]$ is Fredholm of index zero. The kernel of $\partial_\phi \trivialfuncfDP[0, \wavespeed]$ is one-dimensional precisely when there exists a unique $\wavespeed$ such that the equation \eqref{fDP_kernel_dimension} is satisfied, that is
    \begin{equation*}
        \frac{2\pi k}{P} = \sqrt{\biggparanth{3 \frac{\constsol(\wavespeed)}{\wavespeed - \constsol(\wavespeed)}}^{2/s} - 1}.
    \end{equation*}
    The right-hand side of this equation tends to $\sqrt{3^{2/s} - 1}$ when $\wavespeed \rightarrow \infty$. When $\intconst < 0$, the right-hand side is always larger than $\sqrt{3^{2/s} - 1}$, when $\intconst > 0$, the right-hand side is always smaller than $\sqrt{3^{2/s} - 1}$, and equality holds if $\intconst = 0$. Solutions $\wavespeed$ to \eqref{fDP_kernel_dimension} are only possible for the ranges of $P$ and $k$ given in the lemma. For such values of $P$ and $k$, solutions $\bifurcationpoint$ exist and are unique. Note that when $\intconst = 0$, the function is constant, and therefore only satisfied for a single value of $\frac{2\pi k}{P}$.

    For any $(0, \bifurcationpoint)$, the kernel of $\partial_\phi \trivialfuncfDP[0, \wavespeed]$ is one-dimensional and spanned by the function ${\phi_{P, k}^* = \cos(\frac{2\pi k}{P}x)}$. Differentiating $\partial_{\phi} \trivialfuncfDP[0, \bifurcationpoint]$ with respect to the bifurcation parameter $\wavespeed$, one can check that
    \begin{equation*}
        \partial_{\wavespeed \phi} \trivialfuncfDP[0, \bifurcationpoint](\phi_{P, k}^*, 1) = (1 - \constsolpos'(\bifurcationpoint))\phi_{P, k}^* - 3 \constsolpos'(\bifurcationpoint) \bessel{-s} \phi^*_{P, k},
    \end{equation*}
    which belongs to the image of $\partial_{\phi} \trivialfuncfDP[0, \bifurcationpoint]$ if and only if
    \begin{equation*}
        \constsolpos'(\bifurcationpoint) = \frac{\constsolpos(\bifurcationpoint)}{\bifurcationpoint}.
    \end{equation*}
    This is not possible provided $\intconst \neq 0$, and we conclude that the transversality condition holds.
\end{proof}

In contrast to the fKdV equation, Proposition~\ref{prop:local_bifurcation_fDP} shows that for given $s$, local bifurcation for the fDP equation can only happen if the fraction $\frac{2 \pi k}{P}$ is either strictly smaller or strictly larger than $\sqrt{3^{2/s} - 1}$, depending on the parameter $\intconst$. That is, we do not have complete freedom in choosing the period $P$ of solutions. In particular, for $\intconst > 0$ and small $s$ bifurcation only occurs when $P \ll 1$.

From this point on we assume $\intconst > 0$ and consider the local bifurcation branch $\localbifurcationcurvefDP_{P, 1}$ for a fixed period $P$ emanating from the curve $(\constsolpos(\wavespeed), \wavespeed)$ in $\wavespeed_{P, 1}^*$. It is henceforth denoted by $(\varphi(t), \wavespeed(t))$. Furthermore, let
\begin{equation*}
	\submaxsetfDP = \bigset{(\varphi, \wavespeed) \in \holderspaceeven{0, \beta}(\ptorus) \times (\sqrt{\intconst}, \infty) \setsep \varphi < \wavespeed}, \qquad \solutionsetfDP = \bigset{ (\varphi, \wavespeed) \in \submaxsetfDP \setsep \bifurcationfuncfDP(\varphi, \wavespeed) = 0}.
\end{equation*}
The local branch can be parametrized around $t = 0$ in the same way as \eqref{eq:local_branch_parametrization}, only now with $\varphicoeff{0} = \constsolpos(\wavespeed^*) \neq 0$. Moreover, we find that $\varphicoeff{1} = \cos(\frac{2\pi}{P}x)$ and furthermore
\begin{equation*}
	\begin{aligned}
		\varphicoeff{2} & = \frac{1}{3 \constsolpos(\wavespeed^*)} \biggparanth{\frac{1}{\altsymbol(\frac{2\pi}{P})-1} + \frac{1 + 3\altsymbol(\frac{4\pi}{P})}{4 (\altsymbol(\frac{2\pi}{P}) - \altsymbol(\frac{4\pi}{P}))} \cos\bigparanth{\frac{4\pi}{P} x}}, \\
		\wavespeedcoeff{2} & = \frac{1}{3 \constsolpos(\wavespeed^*)} \biggparanth{\frac{1 + 3 \altsymbol(\frac{2\pi}{P})}{\altsymbol(\frac{2\pi}{P}) - 1} + \frac{(1 + 3 \altsymbol(\frac{2\pi}{P}))(1 + 3\altsymbol(\frac{4\pi}{P}))}{8 (\altsymbol(\frac{2\pi}{P}) - \altsymbol(\frac{4\pi}{P}))}}.
	\end{aligned}
\end{equation*}

\begin{proposition} \label{prop:global_bifurcation_fDP}
	For any period $P < 2\pi/ \sqrt{3^{2/s} - 1}$ the local bifurcation branch $(\varphi(t), \wavespeed(t))$ from Proposition~\ref{prop:local_bifurcation_fDP} extends to a global continuous curve ${\globalbifurcationcurvefDP = \set{(\varphi(t), \wavespeed(t)) \setsep t \in [0, \infty)} \subset \submaxsetfDP}$, and one of the following alternatives holds.
	\begin{itemize}
		\item[(i)] $\norm{(\varphi(t), \wavespeed(t))}_{\holderspace{0, \beta} \times \R} \rightarrow \infty$ as $t \rightarrow \infty$,
		\item[(ii)] $\dist(\globalbifurcationcurvefDP, \boundary \submaxsetfDP) = 0$,
		\item[(iii)] $\globalbifurcationcurvefDP$ is a closed loop of finite period.
	\end{itemize}
\end{proposition}

\begin{proof}
	Again we verify the assumptions of \cite[Theorem 9.1.1]{buffoni_toland}. The operator $\partial_{\varphi}\bifurcationfuncfDP[\varphi, \wavespeed]$ is Fredholm of index zero for every $(\varphi, \wavespeed) \in \submaxsetfDP$. Indeed,
    \begin{equation*}
        \partial_\varphi \bifurcationfuncfDP[\varphi, \wavespeed] = (\wavespeed - \phi) \id - 3 \bessel{-s}(\phi\, \cdot\, );
    \end{equation*}
	a sum of the identity and a compact operator. Moreover, any closed and bounded subset of $\solutionsetfDP$ is compact, which can be seen from \eqref{eq:bootstrap_form_fDP} in the same way as before. If we let $\perturbedsubmaxsetfDP$ and $\perturbedsolutionsetfDP$ denote the transformed sets $\submaxsetfDP$ and $\solutionsetfDP$ via $\varphi = \constsolpos(\wavespeed) + \phi$, then both of the above claims hold also for $\partial_\phi \trivialfuncfDP$ in $\perturbedsubmaxsetfDP$ and $\perturbedsolutionsetfDP$.
\end{proof}

Recall the definition \eqref{eq:definition_of_cone} of the cone $\cone$ comprising functions in $\holderspaceeven{0, \beta}(\ptorus)$ which are nondecreasing on the half-period $(\nicefrac{-P}{2}, 0)$. By virtue of Lemma~\ref{lemma:nodal_properties_fdp} one can now prove, in the same way as the proof of Proposition~\ref{prop:exclusion_of_alt_3}, that each solution $\varphi \in \globalbifurcationcurvefDP^1 \cap \cone$ which is also in $\solutionsetfDP^1$ lies in the interior of $\cone$. In view of \cite[Theorem 9.2.2]{buffoni_toland} this allows us to derive the following conclusion.

\begin{proposition} \label{prop:exclusion_of_alt_3_fDP}
    The first component $\varphi(t)$ of the global bifurcation curve $\globalbifurcationcurvefDP$ belongs to $\cone \setminus \{\constsolpos\}$ for all $t > 0$, and  alternative (iii) in Proposition~\ref{prop:global_bifurcation_fDP} does not occur.
\end{proposition}

If the wave speed $\wavespeed$ is bounded along the bifurcation curve, we can find a limiting solution at the end of the curve:

\begin{lemma} \label{lemma:conv_subsequence_global_bifurcation_fDP}
	Any sequence of solutions $(\varphi_n, \wavespeed_n)_{n \in \N} \subset \solutionsetfDP$ to the steady fDP equation with bounded $(\wavespeed_n)_{n \in \N}$ converges uniformly along a subsequence to a solution $(\varphi, \wavespeed)$.
\end{lemma}

\begin{proof}
    Assume that $(\wavespeed_n)_n$ is bounded. Since $\varphi^2 > 0$ we have
    \begin{equation*}
        \norm{\varphi}_{\lebesguespace{\infty}}^2 \leq 2 \intconst + 2 \wavespeed \norm{\varphi}_{\lebesguespace{\infty}},
    \end{equation*}
    so $(\varphi_n)_n$ is bounded. This implies that $(\bessel{-s} \varphi_n^2)_n$ is uniformly equicontinuous ($K_{s}$ is integrable and continuous). So $(\bessel{-s}\varphi_n)_n$ has a uniformly convergent subsequence by Arzela--Ascoli, which also gives a uniformly convergent subsequence for $(\varphi_n)_n$.
\end{proof}

However, for the fDP equation the global bifurcation curve $\globalbifurcationcurvefDP$ is not necessarily bounded in~$\wavespeed$ --- alternative (i) could happen in that $\wavespeed(t) \to \infty$  while $\varphi(t) < \wavespeed(t)$ for all $t > 0$. But for small enough periods, it is possible to exclude this situation.

\begin{proposition} \label{prop:no_sols_large_speeds}
	For sufficiently small periods $P > 0$, the global bifurcation curve ${\globalbifurcationcurvefDP}$ from Proposition~\eqref{prop:global_bifurcation_fDP} is bounded from above in $\wavespeed$. That is, there exists a constant $\overline{\mu} > 0$ such that
	\begin{equation*}
		\sup_{t \in [0, \infty)} \mu(t) \leq \overline{\mu}.
	\end{equation*}
\end{proposition}

\begin{remark}
	The requirement of a sufficiently small period for boundedness in the wave speed, and thereby convergence to a highest wave, is not just a technical condition --- numerical computations indicate that for large values of $P$ (of course still with $P < 2\pi / \sqrt{3^{2/s} - 1}$), the curve does not converge to a highest wave but grows without bound with $\varphi(t) < \wavespeed(t)$ for all $t > 0$.
\end{remark}

\begin{proof}
	Let $(\varphi_n, \wavespeed_n)_{k \in \N}$ be a sequence of solutions along the bifurcation curve $\globalbifurcationcurvefDP$, and assume by contradiction that $\wavespeed_n \to \infty$ as $k \to \infty$. Differentiating the fDP equation for $\varphi_n$ in the form \eqref{eq:main_steady_fdp_linearized} (note that $\varphi_n$ is smooth), multiplying with $\varphi_n'$, taking the absolute value, and integrating over one period gives
	\begin{equation} \label{eq:square_diff_eq}
		\int_{-P/2}^{P/2} \abs{\varphi'_n}^2 \slot \d{x} = \frac{1}{\wavespeed_n} \int_{-P/2}^{P/2} \varphi_n \abs{\varphi'_n}^2 \slot \d{x} + \frac{3}{4} \int_{-P/2}^{P/2} \varphi'_n \altoperator{-s}\varphi'_n \slot \d{x}.
	\end{equation}
	Using Parseval's theorem for the integral in the last term, we find
	\begin{equation*}
		\frac{1}{P} \int_{-P/2}^{P/2} \varphi'_n \altoperator{-s}\varphi'_n \d{x} = \sum_{k \in \Z} a_k \altsymbol\bigparanth{\frac{2\pi k}{P}} a_k \leq \altsymbol\bigparanth{\frac{2\pi}{P}} \sum_{k \in \Z} \abs{a_k}^2 = \altsymbol\bigparanth{\frac{2\pi}{P}} \frac{1}{P} \int_{-P/2}^{P/2} \abs{\varphi'_n}^2 \d{x},
	\end{equation*}
	where $a_k$ are the Fourier coefficients of $\varphi'_n$, and we have used that $a_0 = 0$ since $\varphi'_n$ is odd. From \eqref{eq:square_diff_eq} we now have
	\begin{equation*}
		\int_{-P/2}^{P/2} \abs{\varphi'_n}^2 \slot \d{x} \leq \biggparanth{\frac{1}{\mu_n} \max_{x \in \ptorus} \varphi_n + \frac{3}{4}\altsymbol\bigparanth{\frac{2\pi}{P}}} \int_{-P/2}^{P/2} \abs{\varphi'_n}^2 \slot \d{x}.
	\end{equation*}
	If we can show that the factor in front of the integral on the right-hand side is strictly smaller than $1$ when $\wavespeed_n \rightarrow \infty$ for sufficiently small $P$, this inequality implies $\varphi' \equiv 0$, and we reach a contradiction to Proposition \ref{prop:exclusion_of_alt_3_fDP}. To that end, we recall the first estimate for Hölder regularity in $x = 0$ from Theorem~\ref{thm:regularity_2_fDP}; assuming now that $\varphi_n(0) < \wavespeed_n$, it takes the form
	\begin{equation*}
        \begin{aligned}
			(\varphi_n(0) - \varphi_n(x))^2 & \leq (\wavespeed_n - \varphi_n(x))(\varphi_n(0) - \varphi_n(x)) \leq (2\wavespeed_n - \varphi_n(0) - \varphi_n(x))(\varphi_n(0) - \varphi_n(x)) \\
			& = \frac{3}{2} \int_{\R} (\besselkernel{s}(x + y) + \besselkernel{s}(x - y) - 2 \besselkernel{s}(y)) (\varphi_n^2(0) - \varphi_n^2(y)) \slot \d{y} \\
			& \leq 3 \norm{\varphi_n}_{L^\infty}^2 \int_{\R} \abs{\besselkernel{s}(x + y) + \besselkernel{s}(x - y) - 2 \besselkernel{s}(y)} \slot \d{y} \\
			& \leq C \norm{\varphi_n}_{L^\infty}^2 \abs{x}^s,
        \end{aligned}
    \end{equation*}
	where the final constant only depends on $s$ (via the integral of $K_s$). Note that we have used than $\varphi_n$ is even and nondecreasing on $(\nicefrac{-P}{2}, 0)$ in the above calculation, due to Proposition~\ref{prop:exclusion_of_alt_3_fDP}. This means that
	\begin{equation*}
			\frac{1}{\wavespeed_n} \max_{x \in \ptorus} \varphi_n = \frac{1}{\wavespeed_n} \varphi_n(0) \leq \frac{1}{\wavespeed_n} \varphi_n(x) + \frac{1}{\wavespeed_n} \abs{\varphi_n(0) - \varphi_n(x)} \leq \frac{1}{\wavespeed_n} \varphi_n(x) + \frac{C}{\wavespeed_n} \norm{\varphi_n}_{L^\infty} \abs{x}^{\frac{s}{2}}
	\end{equation*}
	for $x \in \ptorus$. But by Proposition \ref{prop:a_priori_inf_sup_fDP} we know that there is $x_0 \in \ptorus$ such that $\varphi_n(x_0) = \constsolpos$, which implies
	\begin{equation*}
		\frac{1}{\wavespeed_n} \max_{x \in \ptorus} \varphi_n \leq \frac{\constsolpos}{\wavespeed_n} + \frac{C}{\wavespeed} \norm{\varphi_n}_{L^\infty} \abs{x_0}^{\frac{s}{2}} < \frac{1}{2} + \frac{\intconst}{\wavespeed_n^2}+ C P^{\frac{s}{2}},
	\end{equation*}
	where we have used that $\constsolpos / \wavespeed_n \leq 1/2 + \intconst/\wavespeed_n^2$ and $\norm{\varphi_n}_{L^\infty} < \wavespeed$. Consequently, the factor
	\begin{equation*}
		\frac{1}{\mu_n} \max_{x \in \ptorus} \varphi_n + \frac{3}{4}\altsymbol\bigparanth{\frac{2\pi}{P}} \leq \frac{1}{2} + \frac{\intconst}{\wavespeed_n^2}+ C P^{\frac{s}{2}} + \frac{3}{4}\altsymbol\bigparanth{\frac{2\pi}{P}}
	\end{equation*}
	is strictly below $1$ for sufficiently small $P$ as $\wavespeed_n$ becomes large, concluding the proof.
\end{proof}

\newcommand{\constantcurves}[2]{%
    \draw[blue, thick]
    (-4, #1) to[out=0, in=193] (0.5, 0.5) to[out=13, in = 210] (4, #1+1.8);
    \draw[blue, thick]
    (-4, -#1-1.8) to[out=30, in=193] (-0.5, -0.5) to[out=13, in = 180] (4, -#1);
}
\newcommand{\intersections}[1]{%
    \draw[thick, dotted]
    (-#1, -2) to (-#1, 2) coordinate[label = {above:$-\sqrt{\intconst}$}];
    \draw[thick, dotted]
    (#1, -2) to (#1, 2) coordinate[label = {above:$\sqrt{\intconst}$}];
}
\newcommand{\bifurcationrightfdp}[3]{
    \draw[thick, blue] plot[smooth, tension=1]
    coordinates {(#1, #2) (#1 + 0.05, #2 + 0.18 * #3) (#1 + 0.35, #2 + 0.66 * #3) (#1 + 0.35, #2 + #3)};
}
\newcommand{\bifurcationleftfdp}[3]{
    \draw[thick, blue, densely dashed] plot[smooth, tension=1]
    coordinates {(#1, #2) (#1 - 0.05, #2 - 0.18 * #3) (#1 - 0.35, #2 - 0.66 * #3) (#1 - 0.35, #2 - #3)};
}

\begin{figure}[t]
    \centering
    \begin{subfigure}[b]{0.6\textwidth}
        \centering
        \resizebox{0.95\textwidth}{!}{%
        \begin{tikzpicture}
            \draw[line width=0.5em, gray!50, cap=round] plot[smooth, tension=1]
            (-4, 0.15) to[out=0, in=193] (0.44, 0.49);
            \draw[line width=0.5em, gray!50, cap=round] plot[smooth, tension=1]
            (-0.44, -0.49) to[out=13, in = 180] (4, -0.15);
            \draw[thick, black, ->]
            (-4, 0) to (4.2,0) coordinate[label = {right:$\mu$}];
            \draw[very thick, ->]
            (0,-3) to (0,3.5);
            \draw[thick, black]
            (-3, -3) to (3, 3);
            \constantcurves{0.15}{0.5}
            \intersections{0.5}
            \draw
            (3, 3) coordinate[label = {[rotate=45] right:$\mu$}];
            \draw
            (1.5, 2.8) coordinate[label = {$\bm{\max \varphi}$}];
            \draw
            (-1.5, 2.8) coordinate[label = {$\bm{\min \varphi}$}];
            \bifurcationrightfdp{2.4}{1.11}{1.65}
            \bifurcationrightfdp{1.8}{0.875}{1.27}
            \bifurcationrightfdp{1.2}{0.68}{0.87}
            \bifurcationleftfdp{-2.4}{-1.11}{1.65}
            \bifurcationleftfdp{-1.8}{-0.875}{1.27}
            \bifurcationleftfdp{-1.2}{-0.68}{0.87}
        \end{tikzpicture}
        }
        \vspace{0.9cm}
        \caption{Bifurcation diagram}
        \label{fig:bifurcation_sub_diagram_fdp}
    \end{subfigure}
    \begin{subfigure}[b]{0.35\textwidth}
        \begin{subfigure}[b]{\textwidth}
            \centering
            \includegraphics[width=\textwidth]{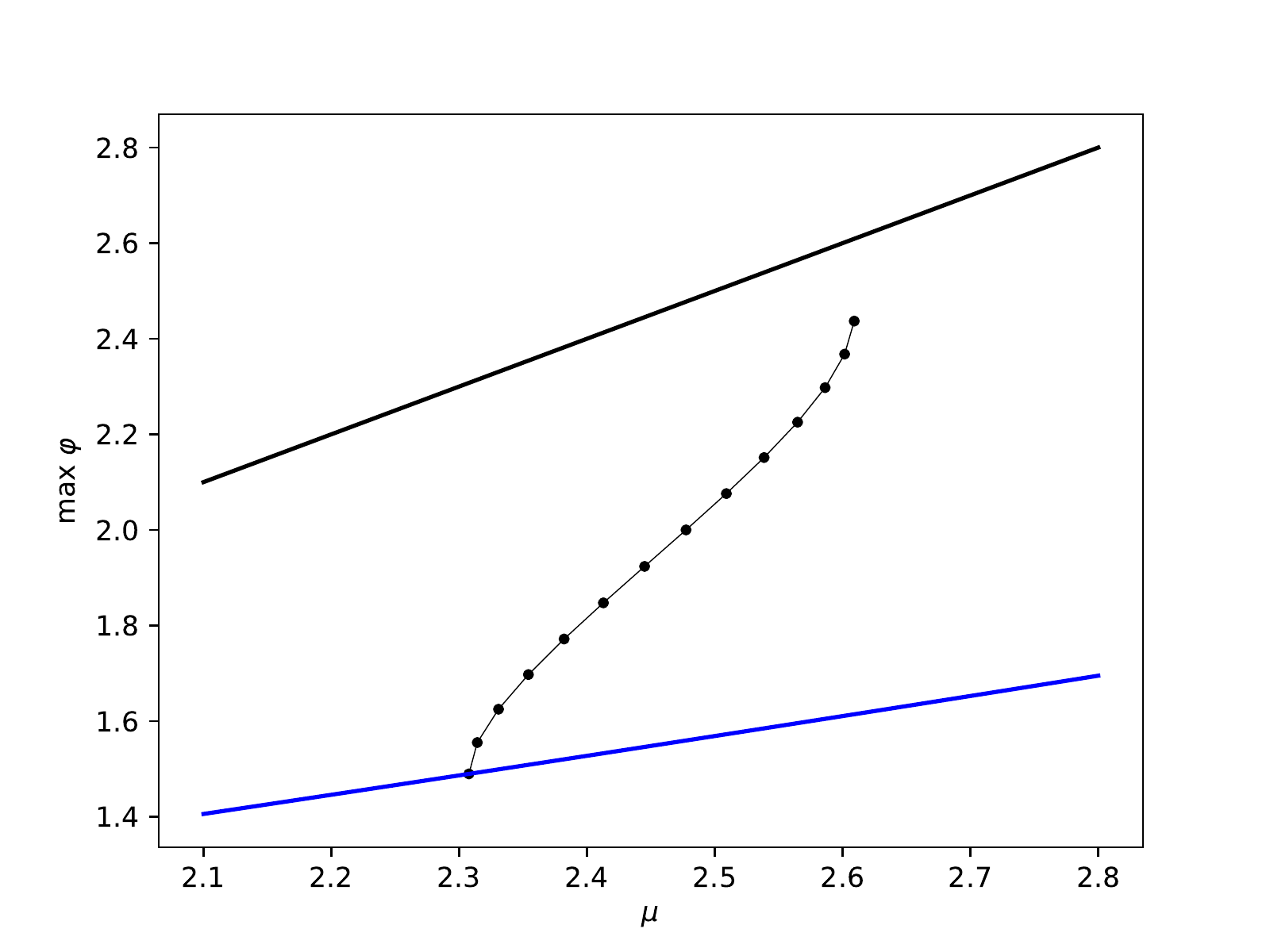}
            \caption{Numerical bifurcation branch}
            \label{fig:branch_fdp}
        \end{subfigure}

        \begin{subfigure}[b]{\textwidth}
            \centering
            \includegraphics[width=\textwidth]{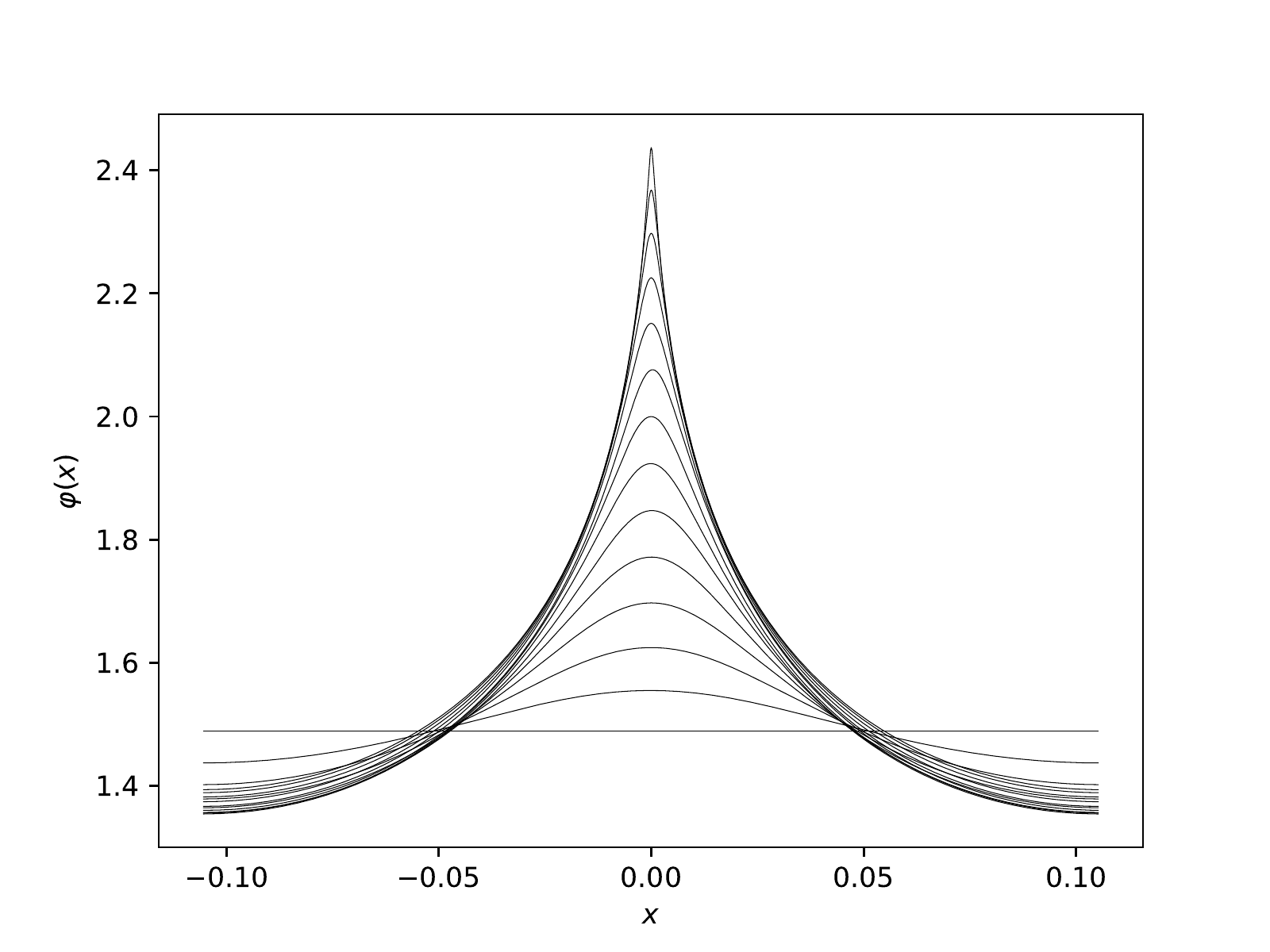}
            \caption{Traveling waves along numerical branch}
            \label{fig:waves_fdp}
        \end{subfigure}
    \end{subfigure}
    \caption{(a) Bifurcation diagram plotting $\max \varphi$ for $\wavespeed > 0$ and $\min \varphi$ for $\wavespeed < 0$ in accordance with the symmetry $(\varphi, \wavespeed) \mapsto (-\varphi(-\cdot), -\wavespeed)$ for the steady fDP equation. Bifurcation branches emanate from the curve $\constsolpos(\wavespeed)$ for $\wavespeed > \sqrt{\intconst}$, and there exist small periods such that local branches extend to a global curves which converges to a highest, cusped traveling wave with ${\varphi(0) = \wavespeed}$. The curves of constant solutions are otherwise locally unique. \mbox{(b)-(c)} Numerical example with $s = 0.5$.}
    \label{fig:bifurcation_diagram_FDP}
\end{figure}

Boundedness of the wave speed in the global bifurcation curve $\globalbifurcationcurvefDP = (\varphi(t), \wavespeed(t))_{t \geq 0}$ is the final ingredient which allows us to prove that there are highest waves for the fDP equation:

\begin{proof}[Proof of Theorem \ref{thm:alt_1_and_2_occur_fDP}]
	We claim that $\wavespeed(t)$ is strictly bounded from below by $\sqrt{\intconst}$. By contradiction, assume that there exists a sequence $(\varphi_n, \wavespeed_n)_n$ with $\wavespeed_n \rightarrow \sqrt{\intconst}$ as $n \rightarrow \infty$. According to Lemma~\ref{lemma:conv_subsequence_global_bifurcation_fDP} we can find a subsequence $(\varphi_n, \wavespeed_n)$ that converges to a solution $(\varphi_0, \wavespeed_0)$. For this subsequence we have
	\begin{equation*}
		\sqrt{\intconst} < \frac{\wavespeed_{n} + \sqrt{\wavespeed_{n}^2 + 8 \intconst}}{4} < \max \varphi_{n} < \wavespeed_n
	\end{equation*}
	owing to Proposition~\ref{prop:a_priori_inf_sup_fDP}. Passing to the limit yields $\max \varphi_0 = \sqrt{\intconst}$ which in turn gives $\max \bessel{-s} \varphi_0^2 = \intconst$. Since $\bessel{-s}$ is strictly monotone, this can only happen if $\varphi_0 \equiv \sqrt{\intconst}$, contradicting Proposition~\ref{prop:upper_regularity_bound_fDP}. Combined with Proposition \ref{prop:no_sols_large_speeds} this means that for every unbounded sequence $(t_n)_{n \in \mathbb{N}}$ of positive numbers, and for sufficiently small periods, there exists a subsequence of $(\varphi(t_n), \wavespeed(t_n))_{n \in \mathbb{N}}$ that converges to a solution $(\varphi, \wavespeed)$ to the steady fDP equation.
	
	Assume now that alternative (i) from Proposition~\ref{prop:global_bifurcation_fDP} occurs but not alternative (ii). Then bootstrapping \eqref{eq:global_bootstrap_fDP} yields $\varphi \in \holderspace{0, \alpha}$ for some $\alpha > \beta$, which is a contradiction. Conversely, assume that alternative (ii) occurs and that $\varphi(t_n)$ remains bounded in $\holderspace{0, \beta}(\R)$. Taking the limit of a subsequence in $C^{0, \beta'}(\R)$ for $s < \beta' < \beta$, the limit must be exactly $s$-Hölder regular at the crest by \eqref{eq:exact_sigma_holder_at_crest}, and we arrive at a contradiction to the boundedness of the sequence in $\holderspace{0, \beta}(\R)$. Consequently, both alternative (i) and (ii) in Proposition~\ref{prop:global_bifurcation_fDP} occur, and the limiting wave $\varphi$ must be even, $P$-periodic, strictly increasing on $(\nicefrac{-P}{2}, 0)$, and exactly \mbox{$s$-Hölder} continuous at $x \in P \mathbb{Z}$.
\end{proof}

\section*{Acknowledgements}
The author thanks M. Ehrnström for guidance and discussions, G. Brüell for suggesting the reformulation of the steady fDP equation to the form \eqref{eq:main_steady_fdp_linearized}, and the anonymous reviewer for pointing out mistakes in an earlier draft and for constructive feedback.

\end{document}